%% file: main.tex
\documentclass[12pt]{article}

\usepackage{amssymb}
\usepackage{amsmath}
\usepackage{amsthm}

\usepackage[hidelinks]{hyperref}

\usepackage{url}
\usepackage{doi}

\input{macros}

\usepackage{tikz}
\usetikzlibrary{quotes}

\usepackage[capitalize]{cleveref}

\theoremstyle{plain}
\newtheorem{theorem}{Theorem}[section]
\newtheorem{lemma}[theorem]{Lemma}
\newtheorem{proposition}[theorem]{Proposition}
\newtheorem{corollary}[theorem]{Corollary} 

\theoremstyle{definition}
\newtheorem{definition}[theorem]{Definition}
\newtheorem{remark}[theorem]{Remark} 
\newtheorem{example}[theorem]{Example} 

  \crefname{fig}{Figure}{Figures}
  \crefname{thm}{Theorem}{Theorems}
  \crefname{lem}{Lemma}{Lemmas}
  \crefname{def}{Definition}{Definitions}
  \crefname{prop}{Proposition}{Propositions}
  \crefname{cor}{Corollary}{Corollaries}
  \crefname{ex}{Example}{Examples}
  \crefname{rem}{Remark}{Remarks}
  \crefname{asm}{Assumption}{Assumptions}

\begin{document}

\title{Quotients-comprehensions Duality in Relational Doctrines}

\author{Francesco Dagnino\thanks{University of Genoa, DIBRIS, Italy. Email: \texttt{francesco.dagnino@unige.it}}
\and
Fabio Pasquali\thanks{University of Milan, Department of Mathematics, Italy. Email: \texttt{fabio.pasquali@unimi.it}}}

\date{} 

\maketitle

\begin{abstract}
We establish a duality between quotients and comprehensions in relational doctrines, a class of indexed posets modelling (a minimal fragment of) the calculus of relations. 
We show that every relational doctrine determines an ``opposite" relational doctrine and that one has quotients if and only if the other has comprehensions. 
This duality extends to the 2-categorical level, yielding a 2-dual isomorphism between the 2-categories of relational doctrines with quotients and that with comprehensions, allowing to derive results on relational doctrines with comprehensions from their dual counterparts for quotients, and viceversa. 
We also study the interaction between quotients and comprehensions and their relative completions, giving sufficient conditions for the induced 2-monads to be composable, \ie sufficient conditions ensuring that applying both the completions, the second one preserves the properties added by the first one.
\end{abstract}

\noindent\textbf{Keywords:} calculus of relations, comprehensions, quotients, duality\\[0.5em]

\bigskip

\input{intro}

\input{intro2}

\input{prelim}

\input{predicates}

\input{duality}

\input{qc-interaction}

\input{conclu}

\bibliographystyle{plainurl} 
\bibliography{biblio}

\end{document}

%% file: macros.tex
\usepackage{etoolbox} 
\usepackage{xcolor} 
\usepackage{xspace}
\usepackage{amsmath}
\usepackage{bbold}
\usepackage[all]{xy}
\usepackage{mathtools}
\usepackage[mathscr]{euscript}

\def\bfd{\begin{color}{orange}}
\def\efd{\end{color}} 

\def\eg{{\em e.g.}\xspace} \def\ie{{\slshape i.e.}\xspace}

\def\ple#1{\ensuremath{{\langle #1 \rangle }}} 

\def\vuoto{}
\newcommand{\FUN}[4]{\ensuremath{{#2}#1{#3}\rightarrow{#4}}}
\newcommand{\fun}[3]{\relax\def\testa{#1}\relax\ifx\testa\vuoto
  \relax\FUN{}{}{{#2}}{{#3}}\else\relax\FUN{:}{{#1}}{{#2}}{{#3}}\fi}

\newcommand{\id}{\mathsf{id}}
\newcommand{\PW}[2][]{\ensuremath{\mathop{\mathscr{P}_{#1}{#2}}}}
\newcommand{\pw}[1]{\relax\def\testa{#1}\relax\ifx\testa\vuoto
  \relax\PW{}\else\relax\PW{\left(#1\right)}\fi}
\newcommand{\fpw}[1]{\relax\def\testa{#1}\relax\ifx\testa\vuoto
  \relax\PW[\omega]{}\else\relax\PW[\omega]{\left(#1\right)}\fi}

\newcommand{\R}{\mathbb{R}} 
\newcommand{\RPos}{\R_{\ge 0}} 

\newcommand{\PP}{\pw{}} 
 
\newcommand{\Rel}{\mathsf{Rel}}

\def\tt{\ensuremath{\textsf{t\kern-.3ex t}}}
\def\ff{\ensuremath{\textsf{f\kern-.3ex f}}}

\let\ForalL\forall \def\Forall#1.{\ForalL_{#1}}
\let\ExistS\exists \def\Exists#1.{\ExistS_{#1}}

\DeclareFontFamily{OT1}{pzc}{}
\DeclareFontShape{OT1}{pzc}{m}{it}{<->s*[1.30]pzcmi7t}{}
\DeclareMathAlphabet{\mathpzc}{OT1}{pzc}{m}{it}
\def\ct#1{\ensuremath{\mathpzc{#1}}}
\def\Ct#1{\ensuremath{\mathbf{#1}}}

\def\op{^{\mbox{\normalfont\scriptsize op}}}

\def\co{^{\mbox{\normalfont\scriptsize co}}}
\newcommand{\bop}[1]{(#1\times#1)\op} 
\def\blank{\mathchoice{\mbox{--}}{\mbox{--}}
{\mbox{\scriptsize--}}{\mbox{\tiny--}}}

\newcommand{\CC}{\ct{C}\xspace}

\newcommand{\D}{\ct{D}\xspace}

\newcommand{\oneAr}[3]{\ensuremath{{#1}:{#2}\rightarrow{#3}}}
\newcommand{\twoAr}[3]{\ensuremath{{#1}:{#2}\Rightarrow{#3}}}
\def\tdot{\textbf{.}}
\def\lsta{\vrule depth4pt width0pt}
\def\lstb{\vrule height5pt width0pt}
\def\arnta[#1]{\ar[#1]|-*=0[@]{\lsta\tdot}}
\def\arntb[#1]{\ar[#1]|-*=0[@]{\lstb\tdot}}
\newcommand{\nt}[3]{\ensuremath{{#1}:{#2} \stackrel{\makebox{\kern-.3ex\tdot}}\rightarrow{#3}}}
\newcommand{\lnt}[3]{\ensuremath{{#1}:{#2} \stackrel{\makebox{\kern-.3ex\tdot}}\rightarrow_l{#3}}}

\newcommand{\Id}{\mathsf{Id}}

\newcommand{\Set}{\ct{Set}\xspace}

\newcommand{\Pos}{\ct{Pos}\xspace}

\newcommand{\EED}{\Ct{EED}\xspace}
\newcommand{\RDtn}{\Ct{RD}\xspace}

\newcommand{\QRDtn}{\Ct{QRD}\xspace}

\newcommand{\ERDtn}{\Ct{ERD}\xspace} 
 
\newcommand{\EQRDtn}{\Ct{EQRD}\xspace}

\newcommand{\order}{\leq} 
\newcommand{\oporder}{\geq} 
\newcommand{\PDoc}{P}

\newcommand{\RDoc}{R}

\newcommand{\SDoc}{S} 
\newcommand{\reidx}[1]{_{#1}}

\newcommand{\fn}[1]{\widehat{#1}}
\newcommand{\lift}[1]{\overline{#1}}

\newcommand{\relr}{\alpha}
\newcommand{\rels}{\beta}
\newcommand{\relt}{\gamma} 
\newcommand{\relz}{\epsilon} 
\newcommand{\eqrelr}{\rho}
\newcommand{\eqrels}{\sigma}
 
\newcommand{\rid}{\mathsf{d}}
\newcommand{\rcomp}{\mathop{\mathbf{;}}}
\newcommand{\rconv}{^{\bot}} 
\newcommand{\rdconv}{^{\bot\bot}} 
\newcommand{\gr}[1]{\Gamma_{#1}}

\newcommand{\QC}[1]{\ensuremath{\ct{Q}_{#1}}\xspace} 
\newcommand{\QR}[1]{({#1})^q}

\newcommand{\QRFun}{\mathrm{U_q}}
\newcommand{\RQFun}{\mathrm{Q}} 
\newcommand{\CRFun}{\mathrm{U_c}}
\newcommand{\RCFun}{\mathrm{C}}
\newcommand{\QMnd}{\mathrm{T_q}}
\newcommand{\CMnd}{\mathrm{T_c}}

\newcommand{\REQFun}{\mathrm{EQ}}
\newcommand{\EQR}[1]{(#1)^{eq}} 
\newcommand{\EQC}[1]{\ensuremath{\ct{EQ}_{#1}}\xspace}

\newcommand{\VRel}[1]{{#1}\text{-}\mathsf{Rel}}

\newcommand{\Qtl}{V}

\newcommand{\Car}[1]{|#1|}
\newcommand{\qord}{\preceq}
\newcommand{\qmul}{\cdot}
\newcommand{\qone}{1} 
\newcommand{\qinv}{\iota} 
\newcommand{\qsup}{\bigvee}

\newcommand{\Sub}{\mathsf{Sub}}

\newcommand{\fmul}{\ast}
\newcommand{\fone}{\kappa}
\newcommand{\finv}{\iota}

\def\RB#1{\mathchoice
  {\rotatebox[origin=c]{180}{$#1$}}
  {\rotatebox[origin=c]{180}{$#1$}}
  {\rotatebox[origin=c]{180}{$\scriptstyle#1$}}
  {\rotatebox[origin=c]{180}{$\scriptscriptstyle#1$}}}
\def\Ex{\RB{E}\kern-.3ex}
\def\Al{\RB{A}\kern-.6ex}

\newcommand{\Span}[1]{\mathsf{Spn}^{#1}}

\newcommand{\spn}[5]{\ensuremath{#1 \xleftarrow{#2} #3 \xrightarrow{#4} #5}}

\newcommand{\vecx}{{\bf x}}
\newcommand{\vecy}{{\bf y}} 

\newcommand{\veczero}{{\bf 0}}

\newcommand{\jmSpan}[1]{\mathsf{JmSpn}^{#1}}

\newcommand{\mnd}{\mathbb{T}}

\newcommand{\xpr}[2][]{\pi_{#2}\ifblank{#1}{}{^{#1}}} 
\newcommand{\fpr}[1][]{\xpr[#1]{1}} 
\newcommand{\spr}[1][]{\xpr[#1]{2}}

\newcommand{\TerAr}{\mathbf{!}}
\newcommand{\DiagAr}{\mathbf{\Delta}}

\newcommand{\PtoR}[1]{\ensuremath{\Rel(#1)}\xspace}

\newcommand{\Pred}[1]{\mathsf{Pred}_{#1}} 
\newcommand{\preda}{\phi}
\newcommand{\predb}{\psi}
\newcommand{\predc}{\xi}
\newcommand{\bdrit}[1]{#1\times#1} 
\newcommand{\OP}{\Omega}
\newcommand{\CRDtn}{\Ct{CRD}\xspace} 
\newcommand{\ECRDtn}{\Ct{ECRD}\xspace} 
\def\co{^{\mbox{\normalfont\scriptsize co}}}

\newcommand{\CR}[1]{(#1)^{c}} 
\newcommand{\CCt}[1]{\ensuremath{\ct{P}_{#1}}\xspace} 
\newcommand{\ECR}[1]{(#1)^{ec}} 
\newcommand{\ECCt}[1]{\ensuremath{\ct{EP}_{#1}}\xspace} 

\newcommand{\img}{\operatorname{Img}}

\newcommand{\PERC}[1]{\ct{PER}_{#1}} 
\newcommand{\PERR}[1]{(#1)^{per}}

%% file: intro.tex
\section{Introduction}
\label{sect:intro}

Starting from the pioneering work by Lawvere \cite{Lawvere69,Lawvere70}, 
indexed posets, or, equivalently, faithful fibrations, often referred to as (hyper)doctrines, 
have been fruitfully adopted in the study of logics and foundations  (see \eg \cite{JacobsB:catltt, PittsCL} and references therein). 
They provide a very general framework enabling an algebraic treatment of fundamental mathematical concepts. 
Specifically, they allow one to identify the minimal structure needed to formulate these concepts, study their properties and interactions, and describe universal constructions to add them to contexts where they are not natively available.

Among other applications, significant effort has been devoted to studying quotients and comprehensions within doctrines \cite{Lawvere70,MaiettiME:eleqc,MaiettiME:quofcm}. 
Intuitively, quotients enable the construction of a new object from an internal equivalence relation, where the relation becomes the equality, while 
 comprehension algebraizes, and generalizes, the separation axiom of set theory, by permitting the formation of subobjects from internal predicates.
These two notions exhibit a certain degree of symmetry.
Indeed, a quotient for an equivalence relation on an object $X$ is an arrow out of $X$ that is initial among those whose kernel contains the given equivalence relation,
while a comprehension for a predicate on an object $X$ is an arrow into $X$ that is terminal among those whose image is contained in the given predicate.
More formally, following \cite{Mellies},
one can organize predicates and equivalence relations into two fibrations:
comprehensions are then given by a right adjoint to a right-adjoint section of the fibration of predicates \cite{JacobsB:catltt,Ehrhard88},
while quotients are given by a left adjoint to a left-adjoint section of the fibration of equivalence relations.
However, this is not enough to make these concepts fully dual to each other:
quotients deal with binary predicates satisfying specific properties (namely, being equivalence relations), whereas comprehensions operate on unary predicates with no further conditions.
So the two concepts are governed by dual categorical structures, but defined on different fibrations that are not straightforwardly linked to one another.

In this paper, we show that the full duality between quotients and comprehensions can be recovered by looking at them through the lens of the calculus of relations \cite{Tarski41,Givant1}.
This means considering a class of doctrines, called relational doctrines \cite{DagninoP23,DagninoP25tac,DagninoP25apal},
which model the core fragment of the calculus of relations: relational identity, composition, and converse.
Our central observation is that
relational doctrines possess an intrinsic duality that is unavailable for doctrines modeling standard predicate logic.
Specifically, given a relational doctrine $\RDoc$, we can construct its opposite $\RDoc\op$, which is again a relational doctrine.
We then prove that $\RDoc$ has quotients for equivalence relations if and only if $\RDoc\op$ has comprehensions for predicates.
This requires identifying predicates in a relational setting with the dual of equivalence relations, \ie coequivalence relations.
As discussed in \cref{sect:intro2}, this choice is less arbitrary than it may appear; it aligns both with definitions in more restrictive settings \cite{FreydS90, Kozen1997Kleene} and with reference examples.
Furthermore, we show that this duality lifts to the 2-categorical level, giving rise to a 2-dual isomorphism between the 2-categories of relational doctrines with quotients and comprehensions, respectively. 
A major consequence of this result is that properties of doctrines with comprehensions can be directly derived from dual properties of doctrines with quotients, 
which have been already studied in the literature \cite{DagninoP25apal,DagninoP26jpaa}, and viceversa.

The rest of the paper is organized as follows.
In \cref{sect:intro2}, we introduce and motivate the choice of identifying predicates with coequivalence relations at the simple propositional level.
\cref{sect:prelim} recalls basic notions and results concerning relational doctrines and quotients, and
in \cref{sect:predicates}, we introduce predicates and comprehensions for relational doctrines.
In \cref{sect:duality}, we describe the duality between quotients and comprehensions in relational doctrines.
By leveraging it, we derive comprehension completions and the associated 2-monadicity results from those for quotients,
characterize them using 
injective objects, and
describe the factorization system induced by comprehensions.
In \cref{sect:q-c-interaction}, we study the interaction between quotients and comprehensions, showing that when both are present, a version of the first isomorphism theorem holds.
Moreover, we show that the completions for quotients and comprehensions do not compose in general, and we provide sufficient conditions for this to be possible.
Finally, \cref{sect:conclu} concludes the paper.

%% file: intro2.tex
\section{Predicates as coequivalence relation: a propositional overview}
\label{sect:intro2}

In this section we introduce and motivate the identification of predicates in the calculus of relations with coequivalence relations. 
To keep things simple, we will work with variants of lattices and relation algebras, that is, we will work at the propositional level. 

Let us start by analysing the reference example for predicates  and relations. 
Given a set $U$, the predicates  on it can be identified with the powerset $\PP(U)$ and the (binary) relations on it with the powerset $\PP(U\times U)$, both ordered by subset inclusion. 
In addition to the standard boolean operations (union, intersection and complement), the poset $\PP(U\times U)$ carries   
a binary operation given by relational composition, denoted by $\relr\rcomp\rels$,  whose neutral element is the identity or diagonal relation $\rid$, and 
a unary operation mapping a relation $\relr$ to its converse $\relr\rconv$. 
The direct image along the diagonal function $\fun{\Delta_U[-]}{\PP(U)}{\PP(U\times U)}$ is monotone and reflects the order (categorically, it is full and faithful). As such, it produces a copy of $\PP(U)$ inside $\PP(U\times U)$. 
Then, one can say that in $\PP(U\times U)$ there are relations that can be regarded as predicates  on $U$. 
The question now is whether we can characterize these relations solely in terms of the relational operations. 
We observe the following properties of these elments:
\begin{enumerate}
\item\label{delta:1} $\Delta[A]$ is coreflexive, \ie $\Delta[A]\subseteq\rid$, 
\item\label{delta:2} $\Delta[A]$ is symmetric, \ie $\Delta[A] = \Delta[A]\rconv$, 
\item\label{delta:3} for every $\relr$ in $\PP(U\times U)$, the relation $\Delta[A]\rcomp\relr\rcomp\Delta[B]$ is the restriction of $\relr$ to $A$ on the left and $B$ on the right and this operation is idempotent, 
\item\label{delta:4} $\Delta[A]\rcomp\Delta[B]$ coincides with the intersection of $\Delta[A]$ and $\Delta[B]$, \ie 
$\Delta[A]\rcomp\Delta[B] = \Delta[A]\cap\Delta[B] = \Delta[A\cap B]$. 
\end{enumerate}
Note that, assuming \cref{delta:1}, \cref{delta:3} is equivalent to the property $\Delta[A]\subseteq\Delta[A]\rcomp\Delta[A]$, which is the dual of transitivity (also known as density or interpolaiton property). 
\cref{delta:3,delta:4} highlight the fact mentioned in \cite{185537} that relational composition is the relational analogue of  conjunction. 
Finally, note that in this case \cref{delta:1} actually implies all the others, but, as we will see, this is due to a special property that is not available in  more general settings. 

In order to work in more generality, let us now recall the algebraic abstraction of the poset $\PP(U\times U)$. 
An \emph{ordered involutive monoid} 
$M = \ple{\Car{M}, \order, \fmul, \fone, \finv}$ 
consists of an ordered monoid $\ple{\Car{M}, \order, \fmul, \fone}$ together with a monotone function \fun{\finv}{\Car{M}}{\Car{M}} such that
$\finv(\fone) = \fone$, $\finv(\relr\fmul\rels) = \finv(\rels)\fmul\finv(\relr)$ and $\finv(\finv(\relr)) = \relr$. 
A \emph{boolean involutive monoid}
is an ordered involutive monoid $M$ where 
$\ple{\Car{M},\order}$ is a boolean algebra, 
$\fmul$ distributes over finite joins and 
$\finv$ preserves boolan operations. 
We report the following result as we have not found it in the literature. 

\begin{proposition}\label{prop:ra-modular}
Let $M$ be a boolean involutive monoid. Then, the following are equivalent: 
\begin{description}
\item[Tarski inequality \cite{Givant1}] $\finv(\relr)\fmul \neg(\relr\fmul\rels) \order \neg\rels$ 
\item[Modularity law \cite{FreydS90}] $(\relr\fmul\rels)\land\relt \order \relr\fmul(\rels \land (\finv(\relr)\fmul\relt))$ 
\end{description}
\end{proposition}
\begin{proof}

Let $\relz=\finv(\relr)\fmul\relt$. By distributivity of joins over the relational operations, it follows that  
$(\relr\fmul \rels)\land\relt=
[(\relr\fmul(\rels\land\relz))\land\relt]\lor [(\relr\fmul (\rels\land\neg\relz))\land\relt]$. Now,  $(\relr\fmul(\rels\land \neg \relz))\land\relt \order
(\relr\fmul\neg \relz)\land\relt = (\finv(\finv(\relr))\fmul \neg(\finv(\relr)\fmul\relt)) \land \relt$. This, by Tarski inequality, is less than or equal to $\neg\relt\land\relt=\bot$, therefore $(\relr\fmul \rels)\land\relt = (\relr\fmul(\rels\land\relz))\land\relt
\order \relr\fmul(\rels\land\relz)
= \relr\fmul(\rels\land\finv(\relr)\fmul\relt)
$

Conversely, by the modular law, it holds that $[\finv(\relr)\fmul \neg(\relr\fmul \rels)]\land \rels\order \finv(\relr)\fmul [(\neg(\relr\fmul\rels))\land (\relr\fmul\rels)]=\finv(\relr)\fmul \bot=\bot$, where the last equality holds again by distributivity of finite joins over the relational operations.
\end{proof}

A \emph{relation algebra} \cite{Givant1} is a boolean involutive monoid satisfying one of the equivalent conditions in \cref{prop:ra-modular}. 
Clearly, $\PP(U\times U)$ with the relational operations described above  is a relation algebra. 

Abstracting the properties observed for the image of $\PP(U)$ into $\PP(U\times U)$, 
we can say that a predicate in an ordered involutive monoid $M$ is an element  $\preda\in\Car{M}$ such that 
$\preda\order\fone$, $\preda\order \finv(\preda)$ and $\preda\order\preda\fmul\preda$. 
By regarding the elements of an ordered involutive monoid as abstract relations, 
the predicates defined above are the dual of equivalence relaitons, \ie coequivalence relations. 
Indeed, by reversing the order, 
the three conditions above become $\fone \order\preda$, $\finv(\preda)\order \preda$ and $\preda\fmul\preda\order\preda$, 
which define an equivalence relation in $M$. 
We write $\Pred{M}$ for the subset of $\Car{M}$ consisting of all predicates in $M$. 
It is immediate to see that 
$\fone$   is the top element of $\Pred{M}$ and 
$\finv$ becomes the identity on $\Pred{M}$. 
Moreover, for all $\preda,\predb\in\Pred{M}$ and $\relt\in\Car{M}$, the function defined by 
$\relt\mapsto\preda\fmul\relt\fmul\predb$ is idempotent. 
Therefore, elements in $\Pred{M}$ satisfy (a generalisation of) \cref{delta:1,delta:2,delta:3}. 
Finally, for all $\preda,\predb,\predc\in\Pred{M}$, we have 
$\predc\order\preda\fmul\predb$ if and only if $\predc\order\preda$ and $\predc\order\predb$. 
This seems to say that $\fmul$ behaves as a meet on $\Pred{M}$. However, this is not the case as the following proposition shows. 

\begin{proposition} \label{prop:iom-pred-comm} 
Let $M$ be an ordered involutive monoid. 
For all $\preda,\predb\in\Pred{M}$, we have 
$\preda\fmul\predb\in\Pred{M}$ if and  only if $\preda\fmul\predb = \predb\fmul\preda$. 
When this holds, $\ple{\Pred{M},\fmul,\fone}$ is a meet-semilaticce. 

\end{proposition}
\begin{proof}
If $\preda\fmul\predb\in \Pred{M}$, we have
$ \preda\fmul\predb 
  = \finv(\preda\fmul\predb) 
  = \finv(\predb)\fmul\finv(\preda) 
  = \predb\fmul\preda$, as needed. 
If $\preda\fmul\predb = \predb\fmul\preda$, we have 
$ \preda\fmul\predb 
  =\predb\fmul\preda
  = \finv(\predb)\fmul\finv(\preda) 
  = \finv(\preda\fmul\predb)$ and 
$ \preda\fmul\predb 
  \order \preda\fmul\preda\fmul\predb\fmul\predb 
  = \preda\fmul\predb\fmul\preda\fmul\predb$, as needed. 
Then, the fact that $\ple{\Pred{M},\fmul,\fone}$ is a meet-semilattice is straightforward. 
\end{proof}

Hence, in general, $\Pred{M}$ is not closed under composition, but, when this happens, it becomes  a meet-semilattice. This happens, for instance, in the degenerate case where $M$ is commutative. 
Another case where this fact holds are relation algebras. 
By \cref{prop:ra-modular}, relation algebras are a special case of one-object allegories \cite{FreydS90}, that we will call \emph{modular involutive monoids}. 
These are ordered involutive monoids with a meet-semilatice structure and where the modularity law holds. 
For these structures we can prove the following 

\begin{proposition}\label{prop:mim}
If $M$ is a modular involutive monoid, then $\Pred{M} = \{ \relr\in\Car{M} \mid \relr\order\fone \}$ and 
$\ple{\Pred{M},\fmul,\fone}$ is a meet-semilattice.
\end{proposition}
\begin{proof}
It suffices to show that $\relr\order\fone$ implies $\relr\order\finv(\relr)$ and $\relr\order\relr\fmul\relr$. 
We have 
$ \relr 
  \order \relr\land\fone 
  \order \relr\fmul(\fone\land\finv(\relr)) 
  = \relr\fmul\finv(\relr)$, which implies 
$\relr\order\finv(\relr)$ and then 
$\relr\order\relr\fmul\relr$. 

To prove that $\ple{\Pred{M},\fmul,\fone}$ is a meet-semilattice, consider $\preda,\predb\in\Pred{M}$ and observe that 
$ \preda\fmul\predb 
  \order (\preda\fmul\predb)\land\fone 
  \order \preda\fmul\predb\fmul(\fone\land\finv(\preda\fmul\predb)) 
  \order \finv(\preda\fmul\predb) 
  = \finv(\predb)\fmul\finv(\preda) 
  = \predb\fmul\preda$,  and so 
by \cref{prop:iom-pred-comm} we get the thesis. 
\end{proof}
 
Indeed, in (one-object) allegories one usually defines predicates as coreflexive relations \cite{FreydS90}. 
However, the fact that all the other desiderata hold is only due to the presence of the modularity law and so,
 they no longer need to hold if we drop it. 
Therefore, this definition does not scale to more general contexts, while our definition via coequivalence relations does, and generalises  the one for allegories. 
In the rest of the paper we will develop this idea in the more general setting of relational doctrines of which ordered involutive monoids are a special case.

%% file: prelim.tex
\section{Preliminaries on relational doctrines}
\label{sect:prelim}

In \cref{sect:intro2} we mentioned relational doctrines as contravariant functors on the product of a category with itself that provide an algebraic tool to model the calculus of binary relations. We recall the definition from \cite{DagninoP25apal}.

\begin{definition}\label[def]{def:rel-doc}
A \emph{relational doctrine} consists of a category \CC and a functor \fun{\RDoc}{\bop\CC}{\Pos} such that for every triple of objects $X,Y,Z$ in \CC there is a monotone function \[\fun{\blank\rcomp\blank}{\RDoc(X,Y)\times\RDoc(Y,Z)}{\RDoc(X,Z)}\] and an element $\rid_X \in \RDoc(X,X)$ such that 

\begin{align*} 
\relr\rcomp(\rels\rcomp\relt) &= (\relr\rcomp\rels)\rcomp\relt 
& 
\rid_X\rcomp\relr &= \relr 
&
\relr\rcomp\rid_Y &= \relr 
\\
(\relr\rcomp\rels)\rconv &= \rels\rconv \rcomp \relr\rconv 
&
\rid_X\rconv &= \rid_X
&
\relr\rdconv &= \relr 
\end{align*} 
and such that for all $\relr\in\RDoc(X,Y)$, $\rels\in\RDoc(Y,Z)$ and 
\fun{f}{A}{X}, \fun{g}{B}{Y} and \fun{h}{C}{Z}  in \CC it holds that
\[\RDoc\reidx{f,g}(\relr)\rcomp\RDoc\reidx{g,h}(\rels) \order \RDoc\reidx{f,h}(\relr\rcomp\rels)
\qquad
\rid_X \order \RDoc\reidx{f,f}(\rid_Y)
\qquad
(\RDoc\reidx{f,g}(\relr))\rconv \order \RDoc\reidx{g,f}(\relr\rconv)\] 
\end{definition}

The element $\rid_X$ is the \emph{identity} or \emph{diagonal} relation on $X$. 
The relation 
$\relr\rcomp\rels$  is the \emph{relational composition} of $\relr$ followed by $\rels$ and   
$\relr\rconv$ is the \emph{converse} of $\relr$.  
Note that all relational operations are lax natural transformations, but 
the operation of taking the converse, being  an involution, is actually strictly natural, \ie $(\RDoc\reidx{f,g}(\relr))\rconv = \RDoc\reidx{g,f}(\relr\rconv)$. 

Given the connection between indexed posets and faithful fibrations via the Grothendieck construction, 
as it is  customary, we adopt some terminology coming from the latter ones: 
for a relational doctrine $\fun{\RDoc}{\bop\CC}{\Pos}$, we refer to $\CC$ as the \emph{base category} (or simply the base) and 
say that $\RDoc$ is a relational doctrine over, or based on, $\CC$, 
the posets of the form $\RDoc(X,Y)$ are called \emph{fibers} over $(X,Y)$, and 
the monotone functions $\RDoc\reidx{f,g} : \RDoc(X',Y') \to \RDoc(X,Y)$ are dubbed \emph{reindexing} along $f:X\to X'$ and $g:Y\to Y'$ in $\CC$.

In a relational doctrine $\fun{\RDoc}{\bop\CC}{\Pos}$, 
every arrow \fun{f}{X}{Y} in \CC induces a relation $\gr{f}=\RDoc\reidx{f,\id_Y}(\rid_Y)\in\RDoc(X,Y)$ 
which is called $\RDoc$-\emph{graph} of $f$, or simply \emph{graph} when the relational doctrines $\RDoc$ is clear from the context. 
Graphs commutes with composition and identities in the sense that $\gr{\id_X} = \rid_X$ and 
$\gr{g\circ f} = \gr{f}\rcomp\gr{g}$. 

An immediate use of graphs is in writing reindexing \fun{\RDoc\reidx{f,g}}{\RDoc(X,Y)}{\RDoc(A,B)} in relational terms and, moreover, in showing that each $\RDoc\reidx{f,g}$ has a left adjoint \fun{\Ex\reidx{f,g}}{\RDoc(A,B)}{\RDoc(X,Y)}. Indeed 
\begin{equation*}\label{left-adj}
\RDoc\reidx{f,g}(\relr) = \gr{f} \rcomp \relr \rcomp \gr{g}\rconv 
\qquad 
\Ex^\RDoc\reidx{f,g}(\rels) = \gr{f}\rconv \rcomp \rels \rcomp \gr{g} 
\end{equation*}
In writing the left adjoints we will omit the reference to $\RDoc$ when this is clear from the context.

Some examples of relational doctrines can be found in \cite{DagninoP25apal,DagninoP23,DagninoP25tac}, 
we recall here  only those that we will use to exemplify the definitions introduced in this paper.

\begin{example}\label[ex]{ex:rel-doc:alg} 
Any  involutive ordered monoid $M=\ple{\Car{M}, \order, \fmul, \fone, \finv}$  can be seen as a relational doctrine based on $\ct{1}$ (the category with only one object and one arrow). It suffices to set 
\[\rid_X = \fone
\qquad 
\relr\rcomp\rels = \relr\fmul\rels
\qquad
\relr\rconv = \finv(\relr) 
\]
In particular, this shows that every relation algebra is a relational doctrine based on $\ct{1}$. 
\end{example}

\begin{example}\label[ex]{ex:rel-doc:vrel} 
Let $\Qtl = \ple{\Car\Qtl,\qord,\qmul,\qone, \qinv}$ be an involutive quantale. 
A \emph{$\Qtl$-relation}  \cite{HofmannST14} between sets $X$ and $Y$ is a function \fun{\alpha}{X\times Y}{\Car\Qtl}. 
The functor \fun{\VRel\Qtl}{\bop\Set}{\Pos} sends a pair of sets
$(X,Y)$ to $\Car\Qtl^{X\times Y}$, the set of $\Qtl$-relations from $X$ to $Y$  ordered pointwise, and a pair of functions \fun{(f,g)}{(X',Y')}{(X,Y)} to the monotone map $\VRel\Qtl\reidx{f,g}$ that evaluates $\fun{\alpha}{X\times Y}{\Car\Qtl}$ to the composition $\fun{\relr(f\times g)}{X'\times Y'}{\Car\Qtl}$. The functor $\VRel\Qtl$ is a relational doctrine, in which
\begin{align*}
\rid_X(x,x') &= \begin{cases}
\qone & x = x' \\
\bot  & x \ne x' 
\end{cases}
\\ 
(\relr\rcomp\rels)(x,z) &= \qsup_{y\in Y} (\relr(x,y)\qmul\rels(y,z))
\\
\relr\rconv(y,x) &= \qinv(\relr(x,y)) 
\end{align*}
for $\relr\in\VRel\Qtl(X,Y)$ and $\rels\in\VRel\Qtl(Y,Z)$. 

Note that every commutative quantale is involutive, taking the identity as the involution. 
Hence, every commutative quantale $\Qtl$ has an associated doctrine of $\Qtl$ relations where the converse is given by 
$\relr\rconv(y,x) = \relr(x,y)$. 
Among other examples, we recall here the relational doctrine of metric relations, which is determined by the commutative quantale
$\RPos = \ple{[0,\infty],\geq,+,0}$ of extended non-negative real numbers, where $+$ is the standard addition extended to $\infty$ by setting 
$\infty + r = r + \infty = \infty$, for all $r \in [0,\infty]$. 

\end{example}

\begin{example}\label[ex]{ex:rel-doc:eed} 
Existential elementary doctrines  \cite{MaiettiME:eleqc}  are functors  $\fun{\PDoc}{\CC\op}{\Pos}$ where $\CC$ has finite products, $\PDoc$ factors through the category of meet-semilattices and homomorphisms of meet-semilattices and for every arrow $f$ in $\CC$ the map $\PDoc(f)$ has a left adjoint satisfying certain distributivity conditions called Frobenius Reciprocity and Beck-Chevalley condition (see \cite{MaiettiME:eleqc} for details). 
Many examples of existential elementary doctrines can be found in \cite{MaiettiME:eleqc,JacobsB:catltt,PittsCL}. 
Existential elementary doctrines model the $(\exists, =,\wedge,\top)$-fragment of predicate logic, also called the regular fragment. One can associate to each elementary existential doctrine a signature having a sort for each object of the base, a function symbol for each morphism, and a relation symbol for each element in the fibres (with obvious arity). The terms and the regular formulas over this signature form the internal language of the doctrine (details are in \cite{PittsCL, JacobsB:catltt}).
We employ this language to sketch how  
any existential elementary doctrine $\PDoc$ on $\CC$ determines a relational doctrine $\fun{\PtoR\PDoc}{\bop\CC}{\Pos}$. The assignments are  $\PtoR\PDoc(X,Y)=P(X\times Y)$ and $\PtoR\PDoc\reidx{f,g}=\PDoc(f\times g)$  
\begin{align*}
(\relr\rcomp\rels)(x,z) &= \exists_{y:Y}(\relr(x,y)\wedge \rels(y,z))
\\
\rid_X(x,x') &= x=_Xx'
\\ 
\relr\rconv(y,x) &= \relr(x,y)
\end{align*}

where $=_X$ is the equality predicate over $X$.

\end{example}

\begin{example}\label[ex]{ex:rel-doc:span} 
Let \CC be a category with weak pullbacks. 
Denote by $\Span\CC(X,Y)$ the poset reflection of the preorder whose objects are spans in \CC between $X$ and $Y$ and where 
$\spn{X}{p_1}{A}{p_2}{Y} \order \spn{X}{q_1}{B}{q_2}{Y}$ if and only if there is an arrow \fun{f}{A}{B} such that 
$p_1 = q_1 \circ f $ and $p_2 = q_2 \circ f$. 
There is a functor  \fun{\Span\CC}{\bop\CC}{\Pos} whose action on arrows is given taking appropriate weak pullbacks. The functor $\Span\CC$ is a relational doctrine where the relational composition  of  
$\relr = \spn{X}{p_1}{A}{p_2}{Y}$ and 
$\rels = \spn{Y}{q_1}{B}{q_2}{Z}$ is obtained considering a weak pullback $p_2$ and $q_1$, while $\rid_X = \spn{X}{\id_X}{X}{\id_X}{X}$ and 
$\relr\rconv=  \spn{Y}{p_2}{A}{p_1}{X}$.  

When \CC is a regular category, for every pair of objects $X,Y$ in $\CC$, 
we can consider the  subposet $\jmSpan\CC(X,Y)$ of $\Span\CC(X,Y)$ on jointly monic spans. 
Then  \fun{\jmSpan\CC}{\bop\CC}{\Pos} is a relational doctrine, where the identity relations and the converses are defined as in \cref{ex:rel-doc:span}, while to compute the relational composition of two jointly monic spans $\relr$ and $\rels$ it suffices to observe that any spans factors through a jointly monic one.
\end{example}

\begin{example}\label[ex]{ex:rel-doc:cob} 
Given a relational doctrine $\fun{\RDoc}{\bop\CC}{\Pos}$, any functor $\fun{F}{\ct{D}}{\ct{C}}$ determines, by  compositing $\RDoc$ with $\fun{\bop{F}}{\bop{\ct{D}}}{\bop\CC}$, a new relational doctrine $\fun{F^\star\RDoc}{\bop{\ct{D}}}{\Pos}$ called \emph{change of base} of $\RDoc$ along $F$.
\end{example}

\begin{remark}\label[rem]{rem:rdoc-base-prod} 
Let us notice that doctrines 
modelling the calculus of predicates, like those presented in \cref{ex:rel-doc:eed}, require a base category with finite products, 
while this is not the case for relational doctrines. 
This difference can be explained by the fact that the calculus of predicates makes an extensive use of variables, while the calculus of relaitons is variable-free, and 
products are essential for dealing with variables, which are represented by projections. 
\end{remark}

 In line with the analogy between doctrines and fibration, we say that a relational doctrine has \emph{fibered finite meets} if every fibre has finite meets and reindexing preserves them.

All the relational doctrines presented in the previous examples have fibered meets except for the relational doctrines of spans $\Span\CC$, 
which have fibered meets provided that $\CC$ has weak finite limits (and not only weak pullbacks). 
Indeed to compute the meet of the relations  represented by $\relr=\spn{X}{p_1}{A}{p_2}{Y}$ and $\rels=\spn{X}{q_1}{B}{q_2}{Y}$ 
one takes a weak limit of the diagram that consists of the two spans
\[
\xymatrix@R=1.5ex@C=1.8ex{
&&L\ar[dl]_-{k}\ar[rd]^-{h}&&\\
&A\ar[dl]_-{p_1}\ar[drrr]&&B\ar[rd]^-{q_2}\ar[llld]&\\
X&&&&Y
}
\]
 The meet  $\relr\land\rels$ is represented by $\spn{X}{p_1k}{L}{q_2h}{Y}$, while the top element is represented by any weak product diagram $\spn{X}{\pi_1}{W}{\pi_2}{Y}$.

Relational doctrines with fibered meets are the common ground for the introduction of next two classes of relational doctrines: the cartesian and the modular ones. These doctrines will be an important reference in \cref{sect:predicates}  for the introduction of the notion of predicate in the relational setting. 

In the next definition, for readability sake, we write $\fpr[X,Y]$ for the first projection $\fun{\pi_1}{X\times Y}{X}$. Similarly, $\spr[X,Y]$ denotes the second projection.

 \begin{definition}\label[def]{def:cartesian}
A relational doctrine \fun\RDoc{\bop\CC}{\Pos} is \emph{cartesian} if it has fibered meet and the following properties hold:
\begin{enumerate} 
\item\label{def:cartesian:1}  \CC has finite products
\item\label{def:cartesian:3}  for all objects $X,Y$ in \CC 
\[ \rid_1 = \top_1 \qquad \rid_{X\times Y} = (\gr{\fpr[X,Y]}\rcomp\gr{\fpr[X,Y]}\rconv)\land(\gr{\spr[X,Y]}\rcomp\gr{\spr[X,Y]}\rconv) \] 
\item\label{def:cartesian:4}  for all $\relr\in\RDoc(A,X)$, $\rels\in\RDoc(X,B)$, $\relr'\in\RDoc(A,Y)$ and $\rels'\in\RDoc(Y,B)$ 
\[ (\relr\rcomp\rels)\land(\relr'\rcomp\rels') = ((\relr\rcomp\gr{\fpr[X,Y]}\rconv)\land(\relr'\rcomp\gr{\spr[X,Y]}\rconv)) \rcomp ((\gr{\fpr[X,Y]}\rcomp\rels)\land(\gr{\spr[X,Y]}\rcomp\rels')) \] 
\item\label{def:cartesian:5}  for all relations $\relr,\rels\in\RDoc(X,Y)$ 
\[ (\relr\land_{X,Y}\rels)\rconv = \relr\rconv\land_{Y,X}\rels\rconv  \qquad \top_{X,Y}\rconv = \top_{Y,X} \] 
\end{enumerate} 
\end{definition}
 
 \cref{def:cartesian:1} of \cref{def:cartesian} asks that the relational doctrine $\RDoc$ is in particular a primary doctrine in the sense of \cite{MaiettiME:eleqc}.
 Conditions from \cref{def:cartesian:3} to \cref{def:cartesian:5} ask that meets are well behaved with respect to the relational operations. Note that, from a logical point of view, the second condition in \cref{def:cartesian:3} corresponds to  the known fact that the equality predicate over the concatenation of two contexts is the conjunction of the equality predicates over each context.

\begin{example} \label[ex]{ex:rel-doc-car} 
The doctrine \fun{\VRel\Qtl}{\bop\Set}{\Pos} of $\Qtl$-relation 
is not cartesian in general. 
Consider the commutative quantale $\ple{[0,\infty], \leq, \cdot, 1}$ where $\cdot$ is the usual multiplication of non-negative real numbers  extended to $\infty$ by setting 
Then, \cref{def:cartesian:3} does not hold because $\rid_1$ is the constant function equal to $1$, while $\top_1$ is the constant function equal to $\infty$. 

 The relational doctrine $\Span\CC$ is cartesian provided that $\CC$ has weak finite limits and  strong binary products. For a regular category $\CC$ the doctrines of the form  $\jmSpan\CC$ are always cartesian. Also the relational doctrines of the form $\PtoR{\PDoc}$ where $\PDoc$ is an existential elementary doctrines are cartesian.
\end{example}

\begin{definition}\label[def]{def:modular}
A relational doctrine $\RDoc$ is \emph{modular} if it has fibered meets and   
for all relations $\relr\in\RDoc(A,X)$, $\rels\in\RDoc(A,Y)$ and $\relt\in\RDoc(X,Y)$ it holds 
\[ \relr\rcomp\relt \land \rels \order (\relr \land \rels\rcomp\relt\rconv)\rcomp\relt \] 
\end{definition}

The terminology in \cref{def:modular} is taken from \cite{FreydS90}.

\begin{example} \label[ex]{ex:rel-doc-mod} 
If $\CC$ has weak finite limits, the doctrine $\Span\CC$ is modular, so is the doctrine $\jmSpan\CC$ of jointly monic spans over a regular category $\CC$. For every existential elementary doctrines $\PDoc$ the relational doctrine $\PtoR{\PDoc}$ is modular, so is $\jmSpan\CC_\ct{M}$. 

The relational doctrine \fun{\VRel\Qtl}{\bop\Set}{\Pos} of $\Qtl$-relation (see \cref{ex:rel-doc:vrel}) is not modular in general. 
Consider for instance the relational doctrine of metric relations, obtained by taking $\Qtl = \RPos$. 
Recall that the order in $\RPos$ is reversed, hence 
$0$ is the top element and 
$r_1\land r_2 = \max\{r_1,r_2\}$. 
Fix a positive real number $r$ and let $\relr=\rels=\relt$ be constant functions equal to $r$. 
Then, we have 
$(\relr\rcomp\relt \land \rels) (a,y) = 2r \land r = 2r$ and  
$((\relr \land \rels\rcomp\relt\rconv)\rcomp\relt) (a,y) = (r\land 2r) + r = 3r$ and, clearly, 
$2r\not\geq3r$. 
\end{example}

Relational doctrines are the objects of the 2-category \RDtn. A 1-arrow \oneAr{F}{\RDoc}{\SDoc} from \fun{\RDoc}{\bop\CC}{\Pos} to \fun{\SDoc}{\bop\D}{\Pos}, 
is a pair \ple{\fn{F},\lift{F}} consisting of 
a functor \fun{\fn{F}}{\CC}{\D} and a natural transformation 
\nt{\lift{F}}{\RDoc}{\SDoc\circ \bop{\fn{F}}},

\[
\xymatrix@C=7.5em@R=1em{
{\bop\CC}\ar[rd]^(.4){\RDoc}_(.4){}="P"
\ar[dd]_{\bop{\fn{F}}}^{}="F"
&\\
 & {\ct{Pos}}\\
{\bop\D}\ar[ru]_(.4){\SDoc}^(.4){}="R"&
\ar"P";"R"_{\lift{F}\kern.5ex\cdot\kern-.5ex}="b"
}
\]
preserving relational identities, composition and converse, that is 
\begin{align*}
\rid_{\fn{F}X} &= \lift{F}_{X,X}(\rid_X)\\ 
\lift{F}_{X,Y}(\relr)\rcomp \lift{F}_{Y,Z}(\rels) &= \lift{F}_{X,Z}(\relr\rcomp\rels)\\ 
\lift{F}_{X,Y}(\relr))\rconv &= \lift{F}_{Y,X}(\relr\rconv)
\end{align*}
for $\relr\in\RDoc(X,Y)$ and $\rels\in\RDoc(Y,Z)$. 

A 2-arrow \twoAr{\theta}{F}{G} is a natural transformation \nt{\theta}{\fn{F}}{\fn{G}} such that 
$\lift{F}_{X,Y} \order \SDoc\reidx{\theta_X,\theta_Y}\circ \lift{G}_{X,Y}$, for all objects $X,Y$ in the base of $\RDoc$
\[
\xymatrix@C=18em@R=1.5em{
{\bop\CC}\ar[rd]^(.4){\RDoc}_(.4){}="P"
\ar@<-1ex>@/_/[dd]_{\bop{\fn{F}}}^{}="F"\ar@<1ex>@/^/[dd]^{\bop{\fn{F'}}}_{}="G"&\\
 & {\ct{Pos}}\\
{\bop\D}\ar[ru]_(.4){\SDoc}^(.4){}="R"&
\ar@/_/"P";"R"_{\lift{F}\kern.5ex\cdot\kern-.5ex}="b"
\ar@<1ex>@/^/"P";"R"^{\kern-.5ex\cdot\kern.5ex \lift{F'}}="c"
\ar"G";"F"_{.}^{\theta\op}\ar@{}"b";"c"|{\le}}
\]

Let us notice that cartesian relational doctrines are precisely the cartesian objects in $\RDtn$. 
Indeed, as observed in
\cite{DagninoP25apal},
the 2-category \RDtn has finite products and a relational doctrine $\RDoc$ is cartesian if and only if 
the 1-arrows $\TerAr_\RDoc$ and $\DiagAr_\RDoc$ have right adjoints in \RDtn.

Similarly, also existential elementary doctrines can be organised into a 2-category \EED where the 1-arrows are the pairs of a product preserving functors between the bases and a natural transformation preserving the relevant structures (see \cite{MaiettiME:quofcm}). 
The assignment presented in \cref{ex:rel-doc:eed} that sends an existential elementary doctrine  $\PDoc$ to the relational doctrine $\PtoR{\PDoc}$ determines a (non-full) inclusion of $\EED$ into $\RDtn$. 
So one can say that a relational doctrine is an existential elementary doctrine if it is in the essential image of that inclusion. The following proposition is proved in \cite{DagninoP25apal}.

\begin{proposition}\label[prop]{prop:eed-iff} 
A relational doctrine $\fun{\RDoc}{\bop\CC}{\Pos}$ is an existential elementary doctrine if and only if it is cartesian and modular.
\end{proposition} 

In the next subsections we will briefly recall some notions and results about quotients and extensional equality in relational doctrines. 
We refer the reader to 
\cite{DagninoP23,DagninoP25apal}. 
for more details.

\subsection{Extensional equality}\label{sect:gs}

Let \fun{\RDoc}{\bop{\CC}}{\Pos} be a relational doctrine.
Recall that 
a relation  $\relr\in\RDoc(X,Y)$ is
\begin{center}
\begin{tabular}{ll}
\emph{functional} if $\relr\rconv \rcomp \relr \order \rid_Y$&\emph{injective} if $\relr\rcomp\relr\rconv \order \rid_X$\\[1ex]
\emph{total} if $\rid_X \order \relr\rcomp\relr\rconv$&\emph{surjective} if $\rid_Y \order \relr\rconv \rcomp \relr$
\end{tabular}
\end{center}

For every arrow $f$ in $\CC$, its graph $\gr{f}$ is functional and total. 
We will say that an arrow $f$ is 
\emph{$\RDoc$-injective} (resp. \emph{$\RDoc$-surjective}) when its graph $\gr{f}$ is injective (resp. surjective). A arrow is \emph{$\RDoc$-bijective} when it is  $\RDoc$-injective and $\RDoc$-surjective.  

Then, recalling that on functional and total relation the order is actually the equality, 
one can easily show that, 
for every pair of parallel arrows $\fun{f,g}{X}{Y}$, the three conditions below are equivalent. 
\[
\rid_X\le \gr{f}\rcomp\gr{g}\rconv
\qquad
\rid_X\le \gr{g}\rcomp\gr{f}\rconv
\qquad
\gr{f}=\gr{g}
\]
These conditions define a notion of equality between the arrows of $\CC$ which is weaker than the actual equality in $\CC$. 
This motivates the following definition.

\begin{definition}\label[def]{def:rel-doc-ex}
A relational doctrine $\fun{\RDoc}{\bop\CC}{\Pos}$ is \emph{extensional} if any two arrows  \fun{f,g}{X}{Y} in \CC are equal whenever $\gr{f} = \gr{g}$. 
\end{definition}

\begin{example}\label[ex]{ex:rel-doc-ex}
The relational doctrines  $\Span\CC$ and $\jmSpan\CC$ of (jointly monic) spans presented in \cref{ex:rel-doc:span} are extensional as $\rid_X$ is $\spn{X}{\id_X}{X}{\id_X}{X}$.
The relational doctrines $\VRel\Qtl$ presented in \cref{ex:rel-doc:vrel} is extensional as for every set $X$ the metric relation $\rid_X$ is two-valued.
The relational doctrines $\PtoR{\PDoc}$ presented in \cref{ex:rel-doc:eed} are extensional if and only if $\PDoc$ has comprehensive diagonals in the sense of \cite{MaiettiME:eleqc}.
\end{example}

As shown in \cite{DagninoP26jpaa}, 
if $\RDoc$ is extensional, then each $\RDoc$-injective arrow is monic and each $\RDoc$-surjective arrow is epic, but the converse is not true in general. 
Furthermore, in general $\RDoc$-bijective arrows, which are both monic and epic, need not to be isomorphisms in general. 

There is a construction that takes a relational doctrine and gives a relational doctrine which is  extensional. The construction is also universal in the sense of the following proposition whose proof is in  \cite{DagninoP25apal} and where $\ERDtn$ denotes the full 2-subcategory of $\RDtn$ on those relational doctrines that are extensional.

\begin{proposition}\label{prop:excol}
The inclusion of $\ERDtn$ into $\RDtn$ has a left biadjoint. 
\end{proposition}

The left biadjoint is obtained taking a relational doctrine $\RDoc$ and quotienting the arrows of the base with respect to the equivalence relation that relates $f$ to $g$ if $\gr{f}=\gr{g}$. The resulting relational doctrine is called the \emph{extensional collapse} of $\RDoc$.

\subsection{Quotients and the quotient completion}
\label{sect:prelim:quot} 

Fix a relational doctrine $\fun{\RDoc}{\bop\CC}{\Pos}$. 
Let $X$ be an object of \CC. 
An \emph{$\RDoc$-equivalence relation} on $X$ is a relation 
$\eqrelr$ in $\RDoc(X,X)$ satisfying
\begin{center}
\begin{tabular}{lll}
reflexivity: $\rid_X\order\eqrelr$& symmetry: $\eqrelr\rconv\order\eqrelr$&transitivity: $\eqrelr\rcomp\eqrelr\order\eqrelr$
\end{tabular}
\end{center}

Every arrow \fun{f}{X}{Y} in \CC induces a $\RDoc$-equivalence relation on $X$, dubbed \emph{kernel} of $f$, given by $\ker_f^\RDoc=\gr{f}\rcomp\gr{f}\rconv$ (we will omit the reference to $\RDoc$ when this is clear from the context).
One immediately sees that for the relational doctrine \fun{\Rel}{\bop\Set}{\Pos} of set-theoretic relations the kernel of a function $f$ is precisely the set $\{(x,x')\mid f(x)=f(x')\}$. 

An $\RDoc$-\emph{quotient arrow} of $\eqrelr$ is an arrow \fun{q}{X}{W} in \CC such that 
$\eqrelr \order \ker_q$ and, 
for every arrow \fun{f}{X}{Z} with $\eqrelr\order\ker_f$, there is a unique arrow \fun{h}{W}{Z} with $f = h\circ q$.  

$\RDoc$-quotient arrows will be often call just \emph{quotient arrows}, when the relational doctrine $\RDoc$ is clear from the context.

We say that a quotient arrow \fun{q}{X}{W}  for $\eqrelr$ is 
\emph{effective} if $\eqrelr = \ker_q$, i.e. its kernel coincides with the equivalence relation, 
and it is \emph{surjective}\footnote{Surjective quotients arrows are also called \emph{descent quotients} such as in \cite{DagninoP25apal}.}
if $\gr{q}\rconv\rcomp\gr{q} = \rid_W$, i.e., $q$ is $\RDoc$-surjective.
Finally, we say that \emph{$\RDoc$ has quotients} if every $\RDoc$-equivalence relation has an effective and surjective $\RDoc$-quotient arrow. 

Any 1-arrow \oneAr{F}{\RDoc}{\SDoc}  in \RDtn 
preserves equivalence relations. 
We say that it 
preserves quotients if, 
whenever $q$ is an $\RDoc$-quotient arrow for an $\RDoc$-equivalence relation $\eqrelr$ on $X$, the arrow $\fn{F}q$ is an $\SDoc$-quotient arrow for the $\SDoc$-equivalence relation $\lift{F}_{X,X}(\eqrelr)$ on $\fn{F}X$. 
We denote by $\QRDtn$ the 2-full 2-subcategory of $\RDtn$ whose objects are relational doctrines with quotients and whose 1-arrows are those of $\RDtn$ that preserves quotients. We recall from \cite{DagninoP25apal} 
the following theorem showing that one can freely complete any relational doctrine with quotients and that these are an algebraic concept which is property-like \cite{BlackwellKP89,Kock95,KellyL97}. 

\begin{theorem}\label[thm]{thm:q-icompletion}
The inclusion 2-functor $\fun{\QRFun}{\QRDtn}{\RDtn}$ has a left biadjoint, 
the composition $\QMnd = \QRFun\circ\RQFun$ is a lax idempotent 2-monad on $\RDtn$ and 
$\QRDtn$ is equivalent to the 2-category of pseudoalgebras for $\QMnd$. 
\end{theorem}

The left biadjoint $\RQFun$ 
sends $\fun{\RDoc}{\bop\CC}{\Pos}$ to \fun{\QR\RDoc}{\bop{\QC\RDoc}}{\Pos}. 

The extensional collapse of a relational doctrine preserves quotients. Thus, if we denote by  $\EQRDtn$ the full 2-subcategory of $\QRDtn$ on those relational doctrines with quotients which are also extensional we get the following result. 

\begin{theorem}\label[thm]{thm:e-q-reflection}
The 2-category $\EQRDtn$ is a reflective 2-subcategory of $\QRDtn$. 
\end{theorem} 

By composing \cref{thm:q-icompletion,thm:e-q-reflection}, we obtain a left biadjoint 
$\fun{\REQFun}{\RDtn}{\EQRDtn}$ to the inclusion of $\EQRDtn$ int $\RDtn$, 
which sends $\fun{\RDoc}{\bop\CC}{\Pos}$ to \fun{\EQR\RDoc}{\bop{\EQC\RDoc}}{\Pos}. It will be called \emph{extensional quotient completion} to be distinguished from the \emph{intensional quotient completion} of \cref{thm:q-icompletion}.

%% file: predicates.tex
\section{Predicates and comprehension}
\label{sect:predicates}

Consider a relational doctrine $\fun{\RDoc}{\bop\CC}{\Pos}$ and an object $A$ of $\CC$, 
we observe that the poset $\RDoc(A,A)$ of endorelations on $A$ carries an ordered involutive monoid structure given by 
relational identity, composition and converse. 
Then, in line with \cref{sect:intro2}, 
 a predicate over $A$ is an element $\preda$ in $\RDoc(A,A)$ which is a coequivalence relation in the sense of the following definition.

\begin{definition}\label[def]{def:predicate}
Given a relational doctrine $\fun{\RDoc}{\bop\CC}{\Pos}$,  and  $A$ an object  of $\CC$, an \emph{$\RDoc$-predicate} over $A$ is a coequivalence relation over $A$, i.e. an element $\preda$ in $\RDoc(A,A)$ satisfying
\begin{center}
\begin{tabular}{lll}
coreflexivity: $\preda\order\rid_A$& (co)symmetry: $\preda\rconv\order\preda$ 
&cotransitivity: 
$\preda\order\preda\rcomp\preda$
\end{tabular}
\end{center}
\end{definition}

We will denote by $\Pred\RDoc(A)$ the subposet of $\RDoc(A,A)$  on $\RDoc$-predicates  over $A$. 
Note that $\ple{\RDoc(A,A), \order, \rcomp, \rid_A, (\blank)\rconv}$ is an ordered involutive monoid, hence, 
by results in \cref{sect:intro2}, 
we have that the elements of  $\Pred\RDoc(A)$, \ie the predicates over $A$, are invariant under the operation of taking the converse, they have $\rid_A$ as top element and they are closed under composition if and only if the composition of two predicates is commutative (in which case $\Pred\RDoc(A)$) becomes an meet-semilattice.
In the following examples we will describe  the posets of predicates in some significant classes of relational doctrines. 

\begin{example}\label[ex]{ex:rel-doc-coeq:alg} 
Recall from \cref{ex:rel-doc:alg} that any ordered involutive monoid $M$ can be regarded as a relational doctrine. 
Then, $\Pred{M}$ is the subposet of $M$ already defined in \cref{sect:intro2}.
\end{example}

\begin{example}\label[ex]{ex:rel-doc-coeq:vrel}
Let $\Qtl = \ple{\Car\Qtl,\qord,\qmul,\qone, \qinv}$ be an involutive quantale and consider the  relational doctrine $\VRel\Qtl$ of $\Qtl$-relations of \cref{ex:rel-doc:vrel}. 
Given a set $A$, 
it is easy to see that a predicat $\preda \in \Pred{\VRel\Qtl}(A)$ is a function 
$\preda : A \to \Pred{\Qtl}$, 
where $\Pred\Qtl$ is defined as in \cref{ex:rel-doc-coeq:alg}. 
When $\Qtl$ is commutative, $\Pred\Qtl$ becomes a complete Heyting algebra, thus $\Pred{\VRel\Qtl}(A)$ becomes a complete Heyting algebra as well. 
In particular, for the quantale $\RPos$ of extended non-negative real numbers, giving rise to the doctrine $\VRel\RPos$ of metric relations,  the poset $\Pred{\VRel\RPos}(A)$ is isomorphic to the powerset of $A$, because $\Pred\RPos = \{0,\infty\}$. 
\end{example}

Modular relational doctrines have a special property: 
predicates over an object $A$ coincide with coreflexive relations on $A$.

\begin{proposition} \label[prop]{prop:pred-modular} 
Let $\fun\RDoc{\bop\CC}{\Pos}$ be a modular relational doctrine and $A$ an object of $\CC$. 
Then, we have 
$\Pred\RDoc(A) = \{ \preda\in\RDoc(A,A) \mid \preda\order\rid_A \}$ and 
$\preda\rcomp\predb = \preda\land\predb$, for all $\preda,\predb\in\Pred\RDoc(A)$. 
\end{proposition}
\begin{proof}
Note that $\RDoc(A,A)$ is a modular involutive monoid. The claim follow by \cref{prop:mim}.
\end{proof}

In particular, in a modular relational doctrine $\RDoc$, the posets $\Pred\RDoc(A)$ are meet-semilattices. 

\begin{example}\label{ex:rel-doc-coeq:alg2}
Let $A$ be a relation algebra, then by \cref{ex:rel-doc:alg} we know that it can be regarded as a relational doctrine. 
We know that $A$ is modular (\ref{prop:ra-modular}).
Then, we have that $\Pred{A} = \{ a \in \Car{A} \mid a\leq \fone \}$ and it is an meet-semilattice (\ref{prop:mim}). 
\end{example}

\begin{example}\label[ex]{ex:rel-doc-coeq:eed}
Recall from \cref{ex:rel-doc:eed} that every existential elementary doctrine $\fun\PDoc{\CC\op}{\Pos}$ induces a relational doctrine 
$\fun{\PtoR\PDoc}{\bop\CC}{\Pos}$ that is cartesian and modular and where 
$\PtoR\PDoc(A,B) = \PDoc(A\times B)$. 
Hence, we have $\Pred{\PtoR\PDoc}(A) = \{ \preda \in \PDoc(A\times A) \mid \preda\order \delta_A \}$, which is an meet-semilattice. 
It is easy to see that $\Pred{\PtoR\PDoc}(A)$ is isomorphic to $\PDoc(A)$. Indeed, 
the map $\Pred{\PtoR\PDoc}(A)\to\PDoc(A)$ is given by reindexing  along the diagonal $\Delta_A$ and 
the map $\PDoc(A) \to \Pred{\PtoR\PDoc}(A)$ is given by its left adjoint 
$\Ex^\PDoc_{\Delta_A}(\phi) = \PDoc\reidx{\pi_1}(\phi)\land \delta_A$ and these are clearly inverse to each other. 
\end{example}

\begin{example}\label[ex]{ex:rel-doc-coeq:span}
Let \CC be a category with weak finite limits, and denote by $\Psi(A)$ the poset reflection of the slice category $\CC/A$ (this is called the poset of weak subobjects in \cite{MaiettiME:eleqc}). 
The poset $\Pred{\Span\CC}(A)$ of $\Span\CC$-predicates over $A$ is (isomorphic to) $\Psi(A)$. 
To see this, observe that, since $\Span\CC$ is modular, the $\Span\CC$-predicates are just the coreflexive relations, i.e. those spans $\spn{A}{f}{X}{g}{A}$ that factors through $\rid_A=\spn{A}{\id_A}{A}{\id_A}{A}$, i.e. such that there is $\fun{k}{X}{A}$ with $k=f$ and $k=g$. Thus the coreflexive relations of $\Span\CC$ are those span whose legs are equal. Similarly, if \CC is regular, the $\jmSpan\CC$-predicates over an object $A$ are the coreflexive relations over   $A$ and these are the subobjects of $A$ (see \cite{FreydS90}).
\end{example}

Having a notion of $\RDoc$-predicate, we can now define when $\RDoc$ has comprehension for them. 
First, notice that every arrow $\fun{f}{A}{B}$ in the base category of a relational doctrine induces a predicate on the codomain $B$, dubbed \emph{image} of $f$, which is $\img_f^\RDoc=\gr{f}\rconv\rcomp\gr{f}$ (as usual, we omit the reference to $\RDoc$ when this is clear from the context). 
Then, given a predicate $\preda$ over an object $A$, its comprehension can be intuitively described as the best underapproximation of $\preda$ as the image of an arrow in the base. 

Let $\fun{\RDoc}{\bop\CC}{\Pos}$ be a relational doctrine and $\preda$ an $\RDoc$-predicate over $A$. 

An arrow $\fun{k}{X}{A}$ is a \emph{weak $\RDoc$-comprehension arrow} for $\preda$ if
\[
\img_k\order\preda
\]
and for every $\fun{f}{Y}{A}$ such that $\img_f\order\preda$ there is an arrow $\fun{h}{Y}{X}$ making the following diagram commute. 
\[
\xymatrix{
X\ar[r]^-{k}&A\\
Y\ar[ru]_-{f}\ar[u]^-{h}&
}
\]
The arrow $k$ is a \emph{strong $\RDoc$-comprehension arrow} if the $h$ is unique or, equivalently, if $k$ is monic.

Let $\fun{k}{X}{A}$ be a (weak/strong) $\RDoc$-comprehension arrow for $\preda$. 
We say that $k$ is \emph{full} if $\img_k = \preda$. 
We say that $k$ is \emph{injective} if $\gr{k}\rcomp\gr{k}\rconv = \rid_X$, that is, $k$ is $\RDoc$-injective. 

\begin{definition}\label[def]{def: predicate}
Let $\fun{\RDoc}{\bop\CC}{\Pos}$ be a relational doctrine. We say that $\RDoc$ has 
\emph{(full/injective/weak/strong) comprehensions} if, every predicate $\preda$ over an object $A$ in \CC  has a (full/$\RDoc$-injective/weak/strong) $\RDoc$-comprehension arrow. 
We say that $\RDoc$ has comprehensions (without any further specification) if it has full injective and strong comprehension.
\end{definition}

\begin{example}\label[ex]{ex:comp:vrel} 
The relational doctrine  \fun{\VRel\Qtl}{\bop\Set}{\Pos} presented in \cref{ex:rel-doc:vrel} has strong injective comprehensions. 
By \cref{ex:rel-doc-coeq:vrel}, we know that a predicate $\preda$ on a set $A$ corresponds to a function 
$\fun\preda{A}{\Pred\Qtl}$. Then, the comprehension of $\preda$ is the injective function 
$\fun{k_\preda}{A_\preda}{A}$ where 
$S_\preda = \{ a \in A \mid \fone \order \preda(a) \}$. 
This works essentially because, for every function $\fun{f}{X}{A}$, its image can only take two values, either $\fone$ or $\bot$, because the diagonal relation takes  only these two values. 
For the same reason, in general these comprehensions are not full. 
However, when $\Pred\Qtl = \{\fone,\bot\}$ the relational doctrine $\VRel\Qtl$ has full injective strong comprehensions. 
This happens for instance  for $\RPos$, where $\Pred\RPos = \{0,\infty\}$. 
\end{example}

The notion of comprehension we have just introduced for relational doctrines 
generalises the corresponding one for existential elementary doctrines as  in \cite{MaiettiME:eleqc}, 
which in turn generalises the one introduced by Lawvere in \cite{Lawvere70}. 
The next example shows how this happens. 

\begin{example}\label[ex]{ex:comp:eed}
Let  $\fun{\PDoc}{\CC\op}{\Pos}$ be an elementary and existential doctrine and  $\fun{\PtoR\PDoc}{\bop\CC}{\Pos}$ the relational doctrine associated with it as in \cref{ex:rel-doc:eed}. 
Recall from \cref{ex:rel-doc-coeq:eed} that $\PtoR\PDoc$-predicates over an object $A$ of \CC bijectively corresponds to $\PDoc(A)$. 
If $\preda$ is an $\PtoR\PDoc$-predicate over $A$, the corresponding element of $\PDoc(A)$ is $\Ex\reidx{\Delta_A}(\preda)$. 
An arrow $\fun{f}{X}{A}$ in \CC
is an $\PtoR\PDoc$-comprehension arrow (resp. weak, full, injective) for $\preda$ if and only if it is a $\PDoc$-comprehension arrow (resp. weak, full, $\PDoc$-injective) for $\Ex\reidx{\Delta_A}(\preda)$ 
in the sense of \cite{MaiettiME:eleqc}. Indeed, $\Ex\reidx{\Delta_A}$ reflects the order, thus for every arrow $\fun{f}{X}{A}$  in \CC, it holds that $\img_f\order\preda$ if and only if $\Ex\reidx{\Delta_A}(\img_f) = \Ex\reidx{f\times f}\Ex\reidx{\Delta_X}(\top)=\Ex\reidx{f}(\top)\order\preda$. Where the last inequality precisely gives the condition with respect to which the comprehension schema is formulated in elementary and existential doctrines.
\end{example}

\begin{example}\label[ex]{ex:comp:span} 

Recall from \cref{ex:rel-doc-coeq:span} that a $\Span\CC$-predicate $\preda$ over $A$ is represented by a span of the form $ \spn{A}{f}{X}{f}{A}$. The arrow $f$ is a weak full $\Span\CC$-comprehension arrow for $\preda$. If \CC is regular, we can reason similarly for $\jmSpan\CC$, except that in this case $f$ is  a strong full and injective $\jmSpan\CC$-comprehension arrow for $\preda$. 
\end{example}

Every 1-arrow of relational doctrines preserves predicates, \ie if \oneAr{F}{\RDoc}{\SDoc} is a 1-arrow from \fun{\RDoc}{\bop\CC}{\Pos} to \fun{\SDoc}{\bop\D}{\Pos} and $\preda$ in $\RDoc(A,A)$ is a  $\RDoc$-predicate, then  $\lift{F}_{A,A}(\preda)$ in $\SDoc(\fn{F}A,\fn{F}A)$ is a  $\SDoc$-predicate.
The relational doctrines with comprehension are the objects of the 2-category $\CRDtn$ which is a 2-full 2-subcategory of $\RDtn$ where the 1-arrows are those 1-arrows of $\RDtn$ that preserve comprehension in the sense that if 
$\fun{k}{X}{A}$ is a $\RDoc$-comprehension arrow for an $\RDoc$-predicate $\preda$ in $\RDoc(A,A)$, then  $\fun{\fn{F}k}{\fn{F}X}{\fn{F}A}$ is a $\SDoc$-comprehension arrow for the $\SDoc$-predicate $\lift{F}_{A,A}(\preda)$.

%% file: duality.tex
\section{Comprehensions are the dual of quotients}
\label{sect:duality}

In this section we address the main goal of the paper: 
we introduce a duality between quotients and comprehensions in reltional doctrines and use it  to derive results for both of these concepts. 
Let us start recalling some notations. 
We have already used $\CC\op$ to denote the opposite category of $\CC$. 
By regarding posets as degenerate categories, 
given a poset  $L = \ple{\Car{L},\le}$, 
we have $L\op = \ple{\Car{L},\geq}$. 
Similarly, if $\Ct{K}$ is a 2-category, 
$\Ct{K}\co$ denotes the category whose objects and 1-arrows are those of $\Ct{K}$ and where 
a 2-arrow $\twoAr{\nu}{f}{g}$ is the same as a 2-arrow $\twoAr{\nu}{g}{f}$ in $\Ct{K}$, that is, 
the hom-category $\Ct{K}\co(A,B)$ is the category $\Ct{K}(A,B)\op$. 
Note that we have $(\CC\op)\op = \CC$ and $(\Ct{K}\co)\co = \Ct{K}$. 
The construction of the opposite category extends to a 2-functor 
$\fun{\mathrm{O}}{\Ct{Cat}\co}{\Ct{Cat}}$ as follows: 
given a category $\CC$, we have $\mathrm{O}(\CC) = \CC\op$, 
given a functor $\fun{F}{\ct{C}}{\ct{D}}$ the functor $\fun{\mathrm{O}(F) = F\op}{\ct{C}\op}{\ct{D}\op}$ acts exactly as $F$ and, 
given functors $\fun{F,G}{\ct{C}}{\ct{D}}$ and a natural transformation $\nt{\nu}{F}{G}$, the natural transformation $\nt{\mathrm{O}(\nu) = \nu\op}{G\op}{F\op}$ has the component
$\fun{\nu\op_X}{GX}{FX}$ in $\ct{D}\op$ that is given by $\fun{\nu_X}{FX}{GX}$ in $\ct{D}$. 
Similarly, given a 2-functor $\fun{F}{\Ct{K}}{\Ct{H}}$, we can define a 2-functor $\fun{F\co}{\Ct{K}\co}{\Ct{H}\co}$. 
Then, it is easy to check that $\mathrm{O}$  is an involution,, that is, $\mathrm{O}\circ\mathrm{O}\co = \Id_{\Ct{Cat}}$. 

We are going to define an involution $\fun{\OP}{\RDtn\co}{\RDtn}$. 
Let $\fun{\RDoc}{\bop\CC}{\Pos}$ be a relational doctrine. 
Recall from \cref{sect:gs} that every map of the form $\RDoc\reidx{f,g}$ has a left adjoint $\Ex\reidx{f,g}$. Therefore the lax-preservation of the relational operations by reindexing (as written in \cref{def:rel-doc}) can be equivalently reformulated as 
\[\Ex\reidx{f,g}(\relr)\rcomp\Ex\reidx{g,h}(\rels) \oporder \Ex\reidx{f,h}(\relr\rcomp\rels)
\qquad
\rid_Y \oporder \Ex\reidx{f,f}\rid_X
\qquad
(\Ex\reidx{f,g}(\relr))\rconv \oporder \Ex\reidx{g,f}(\relr\rconv)\] 

for every $\relr\in\RDoc(A,B)$, $\rels\in\RDoc(B,C)$ and 
\fun{f}{A}{X}, \fun{g}{B}{Y} and \fun{h}{C}{Z}  in \CC. 
The only non trivial inequality is the one concerning the relational composition, 
which can be checked by the following(in)equalities: 
\[
\Ex\reidx{f,h}(\relr\rcomp\rels)=\gr{f}\rconv\rcomp\relr\rcomp\rels\rcomp\gr{h}
=\gr{f}\rconv\rcomp\relr\rcomp\rid_B\rcomp\rels\rcomp\gr{h}\order\gr{f}\rconv\rcomp\relr\rcomp\gr{g}\rcomp\gr{g}\rconv\rcomp\rels\rcomp\gr{h}=\Ex\reidx{f,g}(\relr)\rcomp\Ex\reidx{g,h}(\rels) 
\]
Denote by $\fun{\RDoc\op}{\bdrit\CC}{\Pos}$ the functor whose action on an pair of objects $(X,Y)$ and on a pair of arrows $(f,g)$ is defined as follows
\[
\RDoc\op(X,Y)=\RDoc(Y,X)\op 
\quad
\quad
\RDoc\op\reidx{f,g} = \Ex^\RDoc\reidx{g,f}
\]
Since  $\bdrit\CC$ is $\bop{\CC\op}$ and since the axioms satisfied by the relational operations are equational (so they do not depend on the order) the inequalities displayed above prove that the functor $\RDoc\op$ is a relational doctrine based on $\CC\op$, where the identity relation and the converse of a relation are those of $\RDoc$, while the relational composition of $\relr$ in $\RDoc\op(X,Y)$ with $\rels$ in $\RDoc\op(Y,Z)$ is defined as $\rels\rcomp\relr$. 

 It is immediate to verify that 
$(\RDoc\op)\op = \RDoc$, hence we call the relational doctrine $\RDoc\op$ the \emph{opposite} of $\RDoc$.

Consider now a 1-arrow 
\oneAr{F}{\RDoc}{\SDoc}  where $\SDoc$ is based on $\ct{D}$, 
then we set 
\[
\fn{F\op}=\fun{F\op}{\CC\op}{\ct{D}\op}
\qquad 
\lift{F\op}_{X,Y} = \fun{\lift{F}\op_{Y,X}}{\RDoc(Y,X)\op}{\SDoc(\fn{F}Y,\fn{F}X)\op}
\]
Recall now that, given 1-arrows $\oneAr{F,G}{\RDoc}{\SDoc}$, a 2-arrow \twoAr{\theta}{F}{G} in $\RDtn$ is 
a natural transformation $\nt{\theta}{F}{G}$ such that 
for every $\relr$ in $\RDoc(X,Y)$, $\lift{F}_{X,Y}(\relr) \order \SDoc\reidx{\theta_X,\theta_Y}(\lift{G}_{X,Y}(\relr))$ holds in $\SDoc(\fn{F}X,\fn{F}Y)$ 
or, equivalently, such that $\Ex\reidx{\theta_X,\theta_Y}\lift{F}_{X,Y}(\relr) \order  \lift{G}_{X,Y}(\relr)$. 
This is equivalent to 
\[
 \lift{G}\op_{Y,X}(\relr)\order \RDoc\op\reidx{\theta_Y,\theta_X}\lift{F}_{Y,X}(\relr) 
\]
in $\SDoc\op(\fn{F}Y,\fn{F}X)$. 
Therefore, the natural transformation $\nt{\theta\op}{\fn{G}\op}{\fn{F}\op}$ determines a 2-arrow 
$\twoAr{\theta\op}{G\op}{F\op}$ in $\RDtn$. 

The 2-functor 
$\fun{\OP}{\RDtn\co}{\RDtn}$ 
is then defined by $\OP(\RDoc) = \RDoc\op$, $\OP(F) = F\op$ and $\OP(\theta) = \theta\op$ and it is an involution as proved by the following result. 

\begin{proposition}\label{prop:invo}
The 2-functor $\fun{\OP}{\RDtn\co}{\RDtn}$ is an involution, i.e., $\OP\circ\OP\co = \Id_{\RDtn}$. 
\end{proposition}
\begin{proof}
It follows immediately by noting that $(\RDoc\op)\op = \RDoc$ and  using the same result for 
$\fun{(\blank)\op}{\Ct{Cat}\co}{\Ct{Cat}}$. 
\end{proof} 

Some concepts that one can express in $\RDoc$ take a dual form in $\RDoc\op$. We explore those that will be useful for the duality that we want to study in this section. 
First, we note that graphs are not affected by passing to the dual relational doctrines. 
Indeed, let $\fun{f}{Y}{X}$ be an arrow in $\CC\op$, the $\RDoc\op$-graph of $f$  is the following relation in $\RDoc\op(Y,X)=\RDoc(X,Y)\op$
\[
\RDoc\op\reidx{f,\id_X}(\rid_X) = \Ex^\RDoc\reidx{\id_X,f}(\rid_X) = \gr{\id_X}\rconv\rcomp\rid_X\rcomp\gr{f} = \rid_X\rconv\rcomp\rid_X\rcomp\gr{f} = \gr{f}
\]
where $\gr{f}$ is the $\RDoc$-graph of $f$.

In any category epimorphisms and monomorphisms, that, morally, correspond to surjections and injections, are dual concepts. 
However, when one writes down the conditions of surjectivity and of injectivity  (for example in a first order language) they no longer appear as dual. 
The duality between surjections and injections instead becomes evident using relational doctrines. 

\begin{proposition}\label{prop:inj-to-surj} 
Let $\fun\RDoc{\bop\CC}{\Pos}$ be a relational doctrine. 
A relation $\relr\in\RDoc(X,Y)$ is 
$\RDoc$-surjective if and only if it is $\RDoc\op$-injective.
\end{proposition}
\begin{proof}
We have that 
$\relr$ is $\RDoc$-surjective iff 
$\rid_Y \order \relr\rconv\rcomp\relr$ in $\RDoc(Y,Y)$ iff 
$\relr\rcomp\relr\rconv \order \rid_Y$ in $\RDoc\op(Y,Y)$ iff 
$\relr$ is $\RDoc\op$-injective. 
\end{proof}

By relying on the equality $(\RDoc\op)\op = \RDoc$, the above proposition ensures also that a relation is $\RDoc$-injective if and only if it is $\RDoc\op$-surjective. 
Moreover, since the graphs of arrows in $\RDoc$ and $\RDoc\op$ are given by the same relation,  this duality  between injectivity and surjectivity applies also to arrows. 
Furthermore, a similar duality holds between functional and total relations.

A direct consequence of the definition of $\RDoc\op$ is that a relation is an $\RDoc$-equivalence relation if and only if it  is an $\RDoc\op$-predicate. This helps in showing the duality between another pair of  concepts: the kernel and the image of an arrow $\fun{u}{X}{Y}$. Recall that the kernel of $u$ is  $\ker_u^\RDoc=\gr{u}\rcomp\gr{u}\rconv$ in $\RDoc(X,X)$ and the image of $u$  is $\img_u^\RDoc=\gr{u}\rconv\rcomp\gr{u}$ in $\RDoc(Y,Y)$.

\begin{proposition}\label{prop:imagekerneldual} Let $\fun{\RDoc}{\bop\CC}{\Pos}$ be a relational doctrine. For every $\fun{u}{X}{Y}$ in \CC it holds that 
$\ker_u^\RDoc=\img_u^{\RDoc\op}$.
 \end{proposition}
\begin{proof}
Immediate for the construction of $\RDoc\op$ and the observation  that the graphs in $\RDoc$ and $\RDoc\op$ are the same.
\end{proof}

The duality between kernels and images leads to the following.

\begin{proposition}\label{prop:quot-to-comp}
Let $\fun{\RDoc}{\bop\CC}{\Pos}$ be a relational doctrine and $\relr$ an $\RDoc$-equivalence relation over $X$. 
An arrow $\fun{u}{X}{Y}$ in $\CC$ is an $\RDoc$-quotient arrow for $\relr$ if and only if is an $\RDoc\op$-comprehension arrow for $\relr$.
Moreover, $u$ is an effective $\RDoc$-quotient arrow if and only if it is a full $\RDoc\op$-comprehension arrow.
\end{proposition}

\begin{proof}
Let $u$ be  an $\RDoc$-quotient arrow for $\relr$. Then $\relr \order \img_u^\RDoc$ in $\RDoc(X,X)$ so, by \cref{prop:imagekerneldual} the arrow \fun{u}{Y}{X} in $\CC\op$  and the $\RDoc\op$-predicate $\relr$ is such that  $ \img_U^{\RDoc\op}\order \relr $ in $\RDoc\op(X,X)$. If $\fun{f}{Z}{X}$ in $\CC\op$ is such that $ \img_f^{\RDoc\op}\order \relr $ in $\RDoc\op(X,X)$, then, again by \cref{prop:imagekerneldual}, the arrow $\fun{f}{X}{Z}$ in $\CC$  is such that $ \relr \order \ker_f^{\RDoc} $. By the universal property of quotients, there is a unique arrow $\fun{k}{Y}{Z}$ with $k\circ u=f$. Hence in $\CC\op$ there is a unique arrow $\fun{k}{Z}{Y}$ such that $u\circ k=f$. Similarly one proves the converse. 
Finally 
 $u$, as an $\RDoc$-quotient arrow, is 
effective if and only if $\relr = \ker_u^{\RDoc}$ in $\RDoc(X,X)$.By \cref{prop:imagekerneldual}, this holds if and only if $\relr = \img_u^{\RDoc\op}$ in $\RDoc\op(X,X)$, i.e. if and only if $u$, as an $\RDoc$-comprehension arrow, is full. 
\end{proof}

\cref{prop:quot-to-comp} together with \cref{prop:inj-to-surj} proves the following.

\begin{proposition}\label[prop]{prop:q-c-duality}
Let $\fun{\RDoc}{\bop\CC}{\Pos}$ be a relational doctrine. 
Then, $\RDoc$ has quotients if  and only if $\RDoc\op$ has comprehensions. 
\end{proposition}

Note that, again because $(\RDoc\op)\op = \RDoc$, we also have that $\RDoc$ has comprehensions if and only if $\RDoc\op$ has quotients. 

Recall that two categories $\ct{C}$ and $\Ct{H}$ are dually isomorphic if there is an isomorphism between $\ct{C}$ and $\Ct{H}\op$. 
If  $\Ct{K}$ and $\Ct{H}$  are 2-categories such that  $\Ct{K}\co$  and $\Ct{H}$  are isomorphic, we say that they are \emph{2-dually isomorphic}.

\begin{proposition}\label[prop]{prop:q-c-duality-2-cat}
The 2-categories $\QRDtn$ and $\CRDtn$ are 2-dually isomorphic.
\end{proposition}
\begin{proof}
By \cref{prop:q-c-duality,prop:quot-to-comp}, we have that if $\oneAr{F}{\RDoc}{\SDoc}$ is a 1-arrow in $\QRDtn$, then $\oneAr{F\op}{\RDoc\op}{\SDoc\op}$ is a 1-arrow in $\CRDtn$. Therefore, the 2-functor $\OP$ restricts to a 2-functor 
from $\QRDtn\co$ to $\CRDtn$, which is an isomorphism. 
\end{proof} 

It is also easy to see that extensionality is  a self-dual property. 
Indeed, given a relational doctrine $\RDoc$, 
because in $\RDoc$ and $\RDoc\op$ the identity relations and the graphs are the same, 
we have that 
$\RDoc$ is extensional if and only if $\RDoc\op$ is extensional.
In other words, the functor $\OP$ restricts to an involution on the 2-category $\ERDtn$ of extensional relational doctrines. 
This fact,  together with \cref{prop:q-c-duality-2-cat} proves the following result, where $\ECRDtn$ is the full 2-subcategory of $\CRDtn$ on those relational doctrines with comprehensions that are also extensional. 

\begin{proposition}\label[prop]{prop:e-q-c-duality-2-cat}
The 2-categories $\EQRDtn$ and $\ECRDtn$ are 2-dually isomorphic.
\end{proposition}

 Thanks to \cref{prop:q-c-duality-2-cat,prop:e-q-c-duality-2-cat}  one can immediately derive properties on relational doctrines with comprehension from the corresponding known properties of relational doctrines with quotients. We explore some relevant ones.

\subsection{The comprehension completions}\label{ssect:comp-compl}

Here we describe universal constructions to freely add comprehensions to any relational doctrine, deriving them from the corresponding constructions for quotients.
First, recall that if 
$\fun{\mathrm{R}}{\Ct{K}}{\Ct{H}}$ is a 2-functor with a left biadjoint $\fun{\mathrm{L}}{\Ct{H}}{\Ct{K}}$, then 
$\fun{\mathrm{L}\co}{\Ct{H}\co}{\Ct{K}\co}$ is the left biadjoint of $\fun{\mathrm{R}\co}{\Ct{K}\co}{\Ct{H}\co}$. 
Moreover, if $\fun{\mathrm{O}}{\Ct{K}\co}{\Ct{K}}$ is an involution and $\fun{\mathrm{T}}{\Ct{K}}{\Ct{K}}$ is a (lax idempotent) 2-monad on $\Ct{K}$ whose 2-category of pseudoalgebras is $\Ct{H}$, then 
$\fun{\mathrm{O}\circ\mathrm{T}\co\circ\mathrm{O}\co}{\Ct{K}}{\Ct{K}}$ is a (oplax idempotent) 2-monad on $\Ct{K}$, whose 2-category of pseudoalgebras is $\Ct{H}\co$. 

\begin{proposition}\label[prop]{prop:c-icompletion}
 The inclusion $\fun{\CRFun}{\CRDtn}{\RDtn}$ has a left bi-adjoint $\fun{\RCFun}{\RDtn}{\CRDtn}$, 
the composition $\CMnd = \CRFun\circ\RCFun$ is an oplax idempotent 2-monad on $\RDtn$ and 
$\CRDtn$ is equivalent to the 2-category of pseudoalgebras for $\CMnd$. 
\end{proposition}
\begin{proof}
The first claim follows as  $\CRDtn$ and $\QRDtn$ are 2-dually isomorphic (\cref{prop:q-c-duality-2-cat}), 
the 2-functor $\fun{\QRFun}{\QRDtn}{\RDtn}$  has a left bi-adjoint $\fun{\RQFun}{\RDtn}{\QRDtn}$ (\cref{thm:q-icompletion}), and 
$\fun{\OP}{\RDtn\co}{\RDtn}$ is an involution (\cref{prop:invo}).
Indeed, if $\fun{\mathrm{F}}{\QRDtn\co}{\CRDtn}$ is the isomorphism given by \cref{prop:q-c-duality-2-cat}, 
we have $\RCFun = \mathrm{F}\circ \RQFun\co \circ \OP\co$. 
The other claims follow by noticing that, since $\CRFun\circ\mathrm{F} = \OP\circ\QRFun\co$, we have 
$\CMnd = \CRFun\circ\RCFun = \CRFun\circ\mathrm{F}\circ\RQFun\co\circ\OP\co = \OP \circ \QRFun\co\circ\RQFun\co\circ\OP\co  = \OP\circ\QMnd\co\circ\OP\co$,
again using \cref{thm:q-icompletion,prop:q-c-duality-2-cat}. 
\end{proof}

\cref{prop:c-icompletion} ensures that comprehension is an algebraic concept which is property-like in the sense that every relational doctrine can have at most one algebra structure giving to it comprehension. 
Moreover, this algebra structure is given by a right adjoint in $\RDtn$ to the unit of the 2-monad. 
Note that this is dual to quotients, where the algebra structure is given by a left adjoint.

\begin{proposition} \label[prop]{prop:e-c-reflection}
The 2-category $\ECRDtn$ is a reflective 2-subcategory of $\CRDtn$. 
\end{proposition}
\begin{proof}
By from \cref{thm:e-q-reflection}, we know that $\EQRDtn$ is a reflective 2-subcategory of $\QRDtn$, hence 
$\EQRDtn\co$ is a reflective 2-subcategory of $\QRDtn\co$. 
Then, the thesis follows by 
\cref{prop:q-c-duality-2-cat,prop:e-q-c-duality-2-cat}. 
\end{proof}

Since the reflector of the inclusion of $\QRDtn$ into $\EQRDtn$ is the restriction of the extensional collapse to relational doctrines with quotients, 
the reflector given by \cref{prop:e-c-reflection} is obtained in the same way. 
Hence, in particular, this shows that the extensional collapse preserves comprehensions. 
Finally, by composing \cref{prop:c-icompletion,prop:e-c-reflection}, we get a left biadjoint of the inclusion of $\ECRDtn$ into $\RDtn$. 

In analogy with quotient completions, 
we call the  construction from $\RDtn$ into $\CRDtn$  the \emph{intensional comprehension completion} and
the one from $\RDtn$ into $\ECRDtn$ the \emph{extensional comprehension completion}. 

Note that the two equalities

\[\CMnd   = \OP\circ\QMnd\co\circ\OP\co
\quad
\quad
\QMnd   = \OP\circ\CMnd\co\circ\OP\co
\]

the first of which appears in the proof of \cref{prop:c-icompletion} and the other is its dual analogue, when evaluated on a relational doctrine $\RDoc$, give
\begin{equation}\label{eq:contresempio}
\CR\RDoc = (\QR{\RDoc\op})\op
\quad
\quad
\QR\RDoc = (\CR{\RDoc\op})\op
\end{equation}

That is, the intensional comprehension completion of a relational doctrine $\RDoc$ is the opposite of the intensional quotients completion of the opposite of $\RDoc$, and viceversa. Similar equations hold also for the extensional completions. 

We now give an explicit description of these constructions, starting from the intensional one. 
Given a relational doctrine $\fun{\RDoc}{\bop\CC}{\Pos}$, 
the intensional comprehension completion produces another relational doctrine $\fun{\CR\RDoc}{\bop{\CCt{\RDoc}}}{\Pos}$ defined as follows. 
\begin{itemize}
\item An object of $\CCt\RDoc$ is a pair \ple{X,\preda} of an object $X$ in \CC and an $\RDoc$-predicate $\preda$ on it. 
\item An arrow \fun{f}{\ple{X,\preda}}{\ple{Y,\predb}} in $\CCt\RDoc$ is an arrow \fun{f}{X}{Y} in \CC such that $\preda  \order \gr{f}\rcomp\predb\rcomp\gr{f}\rconv$. 
\item The action of $\CR\RDoc$ on objects is given by 
\[
\CR\RDoc(\ple{X,\preda},\ple{Y,\predb}) = \{ \relr\in\RDoc(X,Y) \mid \relr\order\preda\rcomp\relr\rcomp\predb \}
\]
and the action on arrows \fun{f}{\ple{X,\preda}}{\ple{Y,\predb}} and \fun{g}{\ple{X',\preda'}}{\ple{Y',\predb'}} is given by 
\[
\CR\RDoc\reidx{f,g}(\relr) = \preda\rcomp\RDoc\reidx{f,g}(\relr)\rcomp\preda' 
\]
\item The relational operations of $\CR\RDoc$ are the same as in $\RDoc$ except for the identity which is given by $\rid_{\ple{X,\relr}} = \relr$.
\end{itemize} 
Again in analogy with quotients, 
the elements of the fibre $\CR\RDoc(\ple{X,\preda},\ple{Y,\predb})$ are called \emph{codescent data} for $\preda$ and $\predb$. 

The relational doctrine $\CR\RDoc$ has comprehensions: 
an $\CR\RDoc$-predicate on $\ple{X,\preda}$ is  an $\RDoc$-predicate $\predb$ on $X$ such that $\predb \order \preda$ and 
its full injective comprehension arrow is \fun{\id_X}{\ple{X,\predb}}{\ple{X,\preda}}. 

\begin{example} 
The intensional comprehension completion of a relational doctrine of the form $\PtoR\PDoc$ (as in \cref{ex:rel-doc:eed}) is precisely the comprehension completion of $\PDoc$ described in \cite{MaiettiME:quofcm, MaiettiME:eleqc}.
\end{example}

The extensional quotient completion takes a relational doctrine 
\fun\RDoc{\bop\CC}{\Pos} and return  a new one 
\fun{\ECR\RDoc}{\bop{\ECCt\RDoc}}{\Pos} obtained by applying the extensional collapse to the intensional comprehension completion. 
More explicitly, $\ECR\RDoc$ is essentially the same as $\CR\RDoc$ except for the arrows in the base category. 
Indeed, an arrow in $\ECCt\RDoc$ is an equivalence class of arrows in $\CCt\RDoc$, where 
two arrows \fun{f,g}{\ple{X,\preda}}{\ple{Y,\predb}} are considered equivalent if 
$\preda \order \gr{f}\rcomp\predb\rcomp\gr{g}\rconv$ holds in $\RDoc(X,X)$. 

\begin{example} 
As for the intensional case, the extensional comprehension completion of a relational doctrine of the form $\PtoR\PDoc$ where $\fun{\PDoc}{\CC\op}{\Pos}$ is an elementary and existential doctrine and is precisely the doctrine of predicates as in Theorem 2.7 of \cite{xyz}.
\end{example}

\subsection{Injective objects}\label{ssect:injective}

Relational doctrines obtained via the extensional quotient completion can be characterized using projective objects (see \cite{DagninoP26jpaa}). Here, we dualize this result, identifying those obtained via the extensional comprehension completion using injective objects. Finally, we establish a characterization of the intensional comprehension completion and, by duality, obtain the analogue for the intensional quotient completion.

Let \fun\RDoc{\bop\CC}{\Pos} be a relational doctrine. 
Recall from \cite{DagninoP26jpaa} that an $\RDoc$-projective object is an object in the base of $\RDoc$ that is projective with respect to the class of $\RDoc$-quotient arrows.  
A full subcategory $\ct{G}$ of $\CC$ is an $\RDoc$-projective cover if it contains only $\RDoc$-projective objects and 
every object of $\CC$ can be covered by an object of $\ct{G}$ using an $\RDoc$-quotient arrow. 

We say that an object is \emph{$\RDoc$-injective} if it is injective with respect to the class of $\RDoc$-comprehension arrows. 
More precisely,  an object $I$ of $\CC$ is $\RDoc$-injective if for every diagram of the form
\begin{equation}\label{diag:inj}
\xymatrix{
X\ar[r]^-{f}\ar[d]_{k}&I\\
A&
}
\end{equation} 
if $k$ is a $\RDoc$-comprehension arrow, then there is $\fun{h}{A}{I}$ such that $h\circ k= f$.
A full subcategory $\ct{H}$ of $\CC$ is an \emph{$\RDoc$-injective hull} if it contains only $\RDoc$-injective objects and, 
for every object $X$ of $\CC$, there is an $\RDoc$-comprehension arrow $\fun{m}{X}{I}$ where $I$ is in $\ct{H}$. 
 The relational doctrine $\RDoc$  has \emph{enough $\RDoc$-injectives} if it has an injective hull.

The following proposition states the duality between projectivity and injectivity in the context of relational doctrines. 
It is an immediate consequence of \cref{prop:quot-to-comp}.

\begin{proposition}\label{prop:inj-proj-dual} 
Let \fun\RDoc{\bop\CC}{\Pos} be a relational doctrine. 
An object $X$ in \CC is $\RDoc$-projective if and only if it is $\RDoc\op$-injective. 
A full subcategory $\ct{G}$ of $\CC$ is an $\RDoc$-projective cover if and only if $\ct{G}\op$ is an $\RDoc\op$-injective hull. 
\end{proposition}

Denote by $I_\ct{H}$ the full inclusion of $\ct{H}$ into $\CC$ and by $I_\ct{H}^\star\RDoc$ the change of base of $\RDoc$ along $I_\ct{H}$ (\ie the relational doctrine based on $\ct{H}$ obtained by the composition of $\RDoc$ with $\bop{I_\ct{H}}$ as in \cref{ex:rel-doc:cob}).

\begin{theorem}\label{prop:proj-obj}
Let \fun\RDoc{\bop\CC}{\Pos}  be an extensional relational doctrine with comprehensions and \ct{H} a full subcategory of \CC. 
Then, \ct{H} is an $\RDoc$-injective hull 
if and only if 
$\RDoc$ is equivalent to $\ECR{I_\ct{H}^\star\RDoc}$ in $\ECRDtn$.
\end{theorem}
\begin{proof} 
After \cref{prop:inj-proj-dual} and \cref{prop:q-c-duality}, this is the dual analogue of Corollary 3.8 of \cite{DagninoP26jpaa}.
\end{proof}

\begin{example}\label[ex]{ex:hulls:span} 
Suppose \CC is a regular category. An object of \CC is $\jmSpan\CC$-injective if and only if it is injective in the usual sense, \ie with respect to the class of monomorphisms. So $\jmSpan\CC$ is an extensional comprehension completion if and only if \CC has enough injectives. In particular, the relational doctrine of jointly monic spans over any topos (where the power objects are injectives)  or over any Boolean pretopos (where the objects of the form $X+1$ are injectives) are always  extensional comprehension completions. 
\end{example}

\begin{example}\label[ex]{ex:hulls:fs} 
Suppose \CC has finite limits and $(\ct{E},\ct{M})$ is an orthogonal factorisation system on it, which is also proper, in the sense that the arrows in $\ct{M}$ are monic, and stable in the sense that the pullback of an arrow in $\ct{E}$ is again in $\ct{E}$ \cite{FREYD1972169}.  It determines an elementary existential doctrine over \fun{\Sub_\ct{M}}{\CC\op}{\Pos} taking those subobjects that are represented by arrows in $\ct{M}$ \cite{HughesJacobs2003}. Thus  $\PtoR{\Sub_\ct{M}}$ is a relational doctrine as in \cref{ex:rel-doc:eed}. This relational doctrine has comprehension and if $(\ct{E},\ct{M})$ is such that diagonals are in $\ct{M}$, then  it is also extensional.  The $\PtoR{\Sub_\ct{M}}$-injective objects are those objects of \CC that are injective with respect to the arrows in $\ct{M}$.  Thus $\PtoR{\Sub_\ct{M}}$ arises as an extensioanl comprehension completion if and only if for every $X$ in \CC there is an arrow $\fun{i}{X}{Y}$ of $\ct{M}$ such that $Y$ is injective with respect to $\ct{M}$. A rich list of factorisation systems along with their corresponding injective hulls is presented in \cite{AdamekHerrlichStrecker2009}. 
\end{example}

Sometimes the $\RDoc$-injective hull $\ct{H}$ of \CC is itself is a known category and the restriction of $\RDoc$ to it give a relevant example. Other times, using some know duality, the category $\CC\op$ is of interest as well, and then it arises as a quotient completion (as $\ct{H}\op$ is an $\RDoc\op$-projective over of $\CC\op$).
The following two examples, that are special cases of \cref{ex:hulls:fs} for a specific choice of the factorisation system, exemplify these two situations.

\begin{example}\label[ex]{ex:hulls:top0} 
 Let $\ct{Top}_0$ be the category of $T_0$ topological spaces and continuous functions and consider the factorization system of surjections and subspace inclusions. The doctrine associated to it can be equivalently described as the change of base $U_0^\star\Rel$ of $\fun{\Rel}{\bop\Set}{\Pos}$ along the forgetful functor $\fun{U_0}{\ct{Top}_0}{\Set}$. 
Every $T_0$ space can be seen as subspace of an appropriate  power of the Sierpinski space $\Sigma$.   
These powers of $\Sigma$, together with their retracts, are exactly the $U_0^\star\Rel$-injective objects and thus form  an $U_0^\star\Rel$-injective hull. 
The full subcategory of $\ct{Top}_0$ on these objects  turns out to be equivalent to 
the category of continuous lattices and Scott continuous functions $\ct{Ctn}$ \cite{Gierz2003}. 
Therefore, if $\fun{U_c}{\ct{Ctn}}{\Set}$ is the forgetful functor, then $U_0^\star\Rel$ is the extensional comprehension completion of $U_c^\star\Rel$.

\end{example}

\begin{example}\label[ex]{ex:hulls:inflat} 
Let  $\ct{Lat}_\wedge$ be the category of meet-semilattices and their homomorphisms and consider the factorisation system of surjections and inclusions of meet-semilattices. Similarly to \cref{ex:hulls:top0},  the associated relational doctrine  is $U^\star\Rel$ where $\fun{U}{\ct{Lat}_\wedge}{\Set}$ is the forgetful functor. 
The $U^\star\Rel$-injective objects are the frames and every meet-semilattice can be seen as a sub-meet-semilattice of a frame \cite{BrunsLakser1970}. 
Since   $\ct{Lat}_\wedge$ is dually equivalent to $\ct{ALt}$ (the category of algebraic lattices and functions preserving arbitarry meets and directed joins), 
we get that $\ct{ALat}$ is the base category of the extensional quotient completion of $U^\star\Rel\op$.
\end{example}

In order to provide a characterization of the intensional comprehension completion, 
we introduce \emph{strongly $\RDoc$-injective} objects. 
These are $\RDoc$-injective objects $I$ where, in a diagram like \ref{diag:inj}, 
the arrow $\fun{h}{A}{I}$ is the unique one making it commute. 
Accordingly, a \emph{strong $\RDoc$-injective hull} is an $\RDoc$-injective hull $\ct{H}$ where all objects are strongly $\RDoc$-injective. 
Finally, we say that $\RDoc$ has \emph{enough strong injectives} if it has a strong $\RDoc$-injective hull. 

To prove the characterization (\cref{thm:inj-obj-strong}), we need the following lemmas. 

\begin{lemma} \label[lem]{lem:fib-comp-iso} 
Let \fun\RDoc{\bop\CC}{\Pos} be a relational doctrine with comprehensions, 
If \fun{k}{X}{A} and \fun{k'}{X'}{A'} are $\RDoc$-comprehension arrows for the $\RDoc$-predicates $\preda$ on $A$ and $\preda'$ on $A'$, 
then $\RDoc(X,X')$ is isomorphic to $\CR\RDoc(\ple{A,\preda},\ple{A',\preda'})$. 
\end{lemma}
\begin{proof}
It follows by observing that 
$\RDoc(X,X') = \CR\RDoc(\ple{X,\rid_X},\ple{X',\rid_{X'}})$ and that 
$\fun{\CR\RDoc\reidx{k,k'}}{\CR\RDoc(\ple{A,\preda},\ple{A',\preda'})}{\CR\RDoc(\ple{X,\rid_X},\ple{X',\rid_{X'}})}$ is an isomorphism of posets  because 
$\fun{k}{\ple{X,\rid_X}}{\ple{A,\preda}}$ and $\fun{k'}{\ple{X',\rid_{X'}}}{\ple{A',\preda'}}$ are $\CR\RDoc$-bijective. 
\end{proof}

\begin{lemma}\label[lem]{lem:equiv-inj-hull}
Let $\oneAr{F}{\RDoc}{\SDoc}$ be an equivalence in $\CRDtn$. 
If $\ct{H}$ is a strong $\RDoc$-injective hull, then the image of $\ct{H}$ under $\fn{F}$ is a strong $\SDoc$-injective hull. 
\end{lemma}
\begin{proof}
Since $F$ is an equivalence, it is a right adjoint in $\CRDtn$ and so $\fn{F}$ maps an $\RDoc$-injective objects to an $\SDoc$-injective objects (this is the dual analogous of Proposition 3.5 of \cite{DagninoP26jpaa}). It is also immediate to check that if the object is strong $\RDoc$-injective then its image under $\fn{F}$ is strong $\SDoc$-injective as well.
If $X$ is an object in the base of $\SDoc$, we know that $X\cong\fn{F}Y$ for some object $Y$ in the base of $\RDoc$. 
Because $\ct{H}$ is a strong $\RDoc$-injective hull, 
we have an $\RDoc$-comprehension arrow $\fun{k}{Y}{A}$ where $A$ is an object of $\ct{H}$, thus $\fun{\fn{F}k}{\fn{F}Y}{\fn{F}A}$ is an $\SDoc$-comprehension arrow and precomposing it with the isomorphism $X\cong\fn{F}Y$ we get a $\SDoc$-comprehension arrow $\fun{k'}{X}{\fn{F}A}$ as needed. 
\end{proof}

\begin{theorem}\label[thm]{thm:inj-obj-strong}
Let \fun\RDoc{\bop\CC}{\Pos}  be a relational doctrine with comprehensions. 
Then, $\RDoc$ has enough strong injectives 
if and only if 
$\RDoc$ is equivalent to $\CR\SDoc$ in $\CRDtn$, for some relational doctrine $\SDoc$. 
\end{theorem}
\begin{proof}
Towards a proof of the left-to-right implication, 
let $\ct{H}$ be a strong $\RDoc$-injective hull, write $I_\ct{H}$ for the full inclusion of $\ct{H}$ into $\CC$ and 
$I_\ct{H}^\star\RDoc$ for the change-of-base of $\RDoc$ along $I_\ct{H}$. We prove that $\RDoc$ and $\CR{I_\ct{H}^\star\RDoc}$ are equivalent in $\CRDtn$. 
For every object $\ple{A,\preda}$ in $\CCt{I_\ct{H}^\star\RDoc}$, pick  an $\RDoc$-comprehension arrow $\fun{k_\preda}{A_\preda}{A}$ in $\CC$. 
Then, for every arrow $\fun{f}{(A,\preda)}{(B,\predb)}$ in $\CCt{I_\ct{H}^\star\RDoc}$, consider the following commutative diagram in $\CC$ 
\[
\xymatrix{
  A_\preda 
  \ar[d]_{k_\preda} \ar[r]^-{h_f}
& B_\predb
  \ar[d]^-{k_\predb}\\
  A \ar[r]_-{f} 
& B 
}
\] 
where $h_f$ is uniquely determined by the universal property of $k_\predb$ because 
\begin{align*}
\img_{f\circ k_\preda} 
  &= \gr{f}\rconv\rcomp\gr{k_\preda}\rconv\rcomp\gr{k_\preda}\rcomp\gr{f} 
   = \gr{f}\rconv\rcomp\preda\rcomp\gr{f} 
   \order \predb 
\end{align*}
It is also easy to check that, if $f$ is an $I_\ct{H}^\star\RDoc$-comprehension arrow, then $h_f$ is an $\RDoc$-comprehension arrow for 
$\gr{k_\preda}\rcomp\gr{f}\rconv\rcomp\predb\rcomp\gr{f}\rcomp\gr{k_\preda}\rconv$. 

The assignment mapping $f$ to $h_f$ determines a functor \fun{\fn{F}}{\CCt{I_\ct{H}^\star\RDoc}}{\CC}, 
which is part of a 1-arrow $\oneAr{F}{I_\ct{H}^\star\RDoc}{\RDoc}$ in $\CRDtn$, 
where $\lift{F}$ is a family of isomorphisms given by \cref{lem:fib-comp-iso}. 
Then, to conclude, it suffices to check that $\fn{F}$ is fully faithful and essentially surjective. 
For every object $X$ in $\CC$, there is an $\RDoc$-comprehension arrow $\fun{k}{X}{A}$ where $A$ is an object in $\ct{H}$.
Since comprehensions are full, $k$ is the comprehension of its image $\preda = \gr{k}\rconv\rcomp\gr{k}$. Since $\fun{k_\preda}{A_\preda}{A}$ is another $\RDoc$-comprehension arrow for $\preda$, we have that $X$ and $A_\preda$ are isomorphic in $\CC$, proving that $\fn{F}$ is essentially surjective. 
If $\fun{h}{A_\preda}{B_\predb}$ is an arrow in $\CC$, since $B$ is strongly $\RDoc$-injective, we get a unique  arrow $\fun{f}{A}{B}$ such that 
$f \circ k_\preda = k_\predb \circ h$. 
Notice that 
$ \gr{f}\rconv\rcomp\preda\rcomp\gr{f} 
  = \gr{f}\rconv\rcomp\gr{k_\preda}\rconv\rcomp\gr{k_\preda}\rcomp\gr{f} 
  = \gr{k_\predb}\rconv\rcomp\gr{h}\rconv\rcomp\gr{h}\rcomp\gr{k_\predb}
  \order \gr{k_\predb}\rconv\rcomp\gr{k_\predb}
  = \predb$, 
since both $k_\preda$ and $k_\predb$ are full. 
Therefore, $\fun{f}{\ple{A,\preda}}{\ple{B,\predb}}$ is a well-defined arrow in $\CCt{I_\ct{H}^\star\RDoc}$ and so $h = h_f$, proving that $\fn{F}$ is fully faithful. 

For the other direction, 
consider a relational doctrine $\fun\SDoc{\bop{\ct{D}}}{\Pos}$ together with an equivalence 
$\oneAr{F}{\CR\SDoc}{\RDoc}$ in $\CRDtn$. 
First, note that objects of the form $\ple{A,\rid_A}$ (with $A$ an object of $\ct{D}$) constitute a strong $\CR\SDoc$-injective hull. 
Indeed, $\CR\SDoc$-comprehension arrows are supported by the identities of $\ct{D}$, 
hence, given $\fun{f}{\ple{B,\preda}}{\ple{A,\rid_A}}$ and $\fun{\id_B}{\ple{B,\preda}}{\ple{B,\predb}}$, the arrow $\fun{g}{\ple{B,\predb}}{\ple{A,\rid_A}}$ exists and is necessarily equal to $f$. 
Then, the thesis follows by \cref{lem:equiv-inj-hull}. 
\end{proof}

The definition of strongly $\RDoc$-injective object and that of strong $\RDoc$-injective hull dualise in the obvious way in the definitions of strong $\RDoc$-projective object and that of strong $\RDoc$-projective cover. Thus \cref{thm:inj-obj-strong} as the following corollary that follows immediately by duality.
 
\begin{corollary}\label[cor]{cor:proj-obj-strong-cor}
Let \fun\RDoc{\bop\CC}{\Pos}  be a relational doctrine with quotients. 
Then, $\RDoc$ has enough strong  projectives 
if and only if 
$\RDoc$ is equivalent to $\QR\SDoc$ in $\QRDtn$, for some relational doctrine $\SDoc$. 
\end{corollary}

\begin{example}
The category $\ct{SVec}$ of semi-normed vector spaces\footnote{Here we consider semi-norms  that take values in $[0,\infty]$.}  and short maps is the base of the relational doctrine $\mathsf{SN}$ sending two semi-normed vector spaces $X$ and $Y$ to the set of functions $\fun{\alpha}{\Car{X}\times\Car{Y}}{[0,\infty]}$ ordered point-wise and such that 
\begin{align*} 
\alpha(\vecx,\vecy)+\alpha(\vecx',\vecy') &\geq \alpha(\vecx+\vecx',\vecy+\vecy') \\ 
\alpha(a\cdot \vecx, a\cdot\vecy)  &= |a|\cdot \alpha(\vecx,\vecy) \\ 
\|\vecx\|_X + \|\vecy\|_Y &\geq \alpha(\vecx,\vecy) 
\end{align*} 
The first two condition state that $\alpha$ is a semi-norm on the product space $\Car{X}\times\Car{Y}$  and the third one requires $\alpha$ to be bounded by the  sum of the semi-norms of the two spaces. 
It is shown in \cite{DagninoP25apal} that this is a relational doctrine where relations compose like those in $\VRel\RPos$ and the identity on $X$ is given by 
$\rid_X(\vecx,\vecx') = \|\vecx-\vecx'\|_X$. 
Moreover, it has quotients: if $\eqrelr$ is an $\mathsf{SN}$-equivalence relation on $X$, its quotient is the arrow 
$\fun{\id_{\Car{X}}}{X}{\ple{\Car{X},\|\blank\|_\eqrelr}}$ where 
$\|\vecx\|_\eqrelr = \eqrelr(\vecx,\veczero)$. 
The category $\ct{Vec}$ of vector spaces and linear maps is a full subcategory of $\ct{SVec}$ via the nclusion assigning to any a vector space  the semi-norm that maps $\veczero$ to $0$ and the other vectors to $\infty$. 
It is an easy check that vector spaces, via this inclusion, form a strong $\mathsf{SN}$-projective cover, 
thus showing that $\mathsf{SN}$ is the intensional quotient completion of its restriction to $\ct{Vec}$. 
\end{example}

The previous results trace a connection between injective objects and comprehension dualising the known connection between projective objects and quotients, of which the characterisation of the exact categories that are an ex/wlex completion is an instance. 
Taking inspiration from 
the ex/wlex completion, in \cite{DagninoP26jpaa} we generalised to relational doctrines the key result  about the exact completion that monadic categories over $\Set$ are the ex/wlex completion of their free algebras (see  \cite{Vitale94}). 
We conclude the section writing a dual version of this theorem.

We first show how to construct doctrines of co-algebras for a co-monad in the 2-category $\RDtn$.
Let \fun{\RDoc}{\bop\CC}{\Pos}  be a relational doctrine and $\mnd = \ple{T,\epsilon,\delta}$ a co-monad on $\RDoc$ in $\RDtn$ (see \cite{Street72} for  the general 2-categorical definition and \cite{DagninoR21} for the instantiation on doctrines). 
We can define the 
doctrine \fun{\RDoc^\mnd}{\bop{\CC^{\fn\mnd}}}{\Pos} where the base is the category of co-algebras for the co-monad 
$\fn{\mnd} = \ple{\fn{T},\epsilon,\delta}$ on \CC and where the relations  from a co-algebra \ple{X,a} to a co-algebra \ple{Y,b} are defined as follows
\[
\RDoc^\mnd(\ple{X,a},\ple{Y,b})=\{\relr\in\RDoc(X,Y)\mid \relr\order\gr{a}\rcomp\lift{T}_{X,Y}(\relr)\rcomp\gr{b}\rconv\}
\]
Note that it always holds that  $\gr{a}\rcomp\lift{T}_{X,Y}(\relr)\rcomp\gr{b}\rconv\order\relr$ so the relation over the two co-algebras \ple{X,a} and \ple{Y,b} are the fixed points of the operator $\gr{a}\rcomp\lift{T}_{X,Y}(\relr)\rcomp\gr{b}$. In other words, these are bisimulations between the two coalgebras \cite{Jacobs16}.
The action on arrows is the restriction of that of $\RDoc$, that is $\RDoc^\mnd_{f,g}=\RDoc_{f,g}$. 

For a full subcategory $\D$ of $\CC$, we denote by $\D_{\fn{\mnd}}$ the full subcategory of the  category of co-algebras $\CC^{\fn\mnd}$ spanned by co-free co-algebras generated by objects in $\D$, i.e., co-algebras of the form \ple{\fn{T}X,\delta_X} for $X$ and object of \D. 
We also write \fun{K_\D}{\D_{\fn\mnd}}{\CC^{\fn\mnd}} for the inclusion functor.
Note that when \D coincides with \CC, the category $\D_{\fn\mnd}$ is the category of co-free co-algebras of  the co-monad $\fn\mnd$, sometimes called co-Kleisli category. 

By abstract results about (co)monads in a 2-category \cite{Street72}, we know that 
$\mnd = \ple{T,\epsilon.\delta}$ is a comonad on $\RDoc$ in $\RDtn$ if and only if 
$\mnd\op = \ple{T\op,\epsilon\op,\delta\op}$ is a monad on $\RDoc\op$ in $\RDtn$. 
Then, as a consequence of \cref{prop:q-c-duality}, we get the following results, which are the dual of the analogous ones for quotients (see \cite[Theorem 5.8]{DagninoP25tac} and \cite[Theorem 4.4]{DagninoP26jpaa}).

\begin{proposition}\label{prop:coalrel}
Let $\RDoc$ be a relational doctrine with comprehensions and $\mnd$ a comonad on it in $\CRDtn$. 
Then, $\RDoc^\mnd$ has comprehensions and, if $\RDoc$ is also extensional, $\RDoc^\mnd$ is extensional as well. 
\end{proposition}

\begin{theorem}\label[thm]{thm:eqc-comnd} 
Let \fun\RDoc{\bop\CC}{\Pos} be an extensional relational doctrine with comprehension and $\ct{H}$ a subcategory of \CC. 
The following are equivalent. 
\begin{enumerate}
\item\label{thm:eqc-mnd:1}
\ct{H} is an $\RDoc$-injective hull.
\item\label{thm:eqc-mnd:2}
For every co-monad $\mnd = \ple{T,\epsilon,\delta}$ on $\RDoc$ in \ECRDtn, $\ct{H}_{\fn\mnd}$ is an $\RDoc^\mnd$-injective hull.
\end{enumerate}
\end{theorem} 

Observe that the co-monad $\mnd$ of \cref{thm:eqc-comnd} is in \ECRDtn, this means that the underlying 1-arrow preserves comprehension. 

\begin{example} 
Consider the relational doctrine  $\fun{\Rel}{\bop\Set}{\Pos}$ of set theoretic relations. Here a $\Rel$-predicate over a set $A$ is just a subset of $A$ and $\Rel$-comprehension arrows are the injective functions to $A$. So an 1-arrow \oneAr{F}{\Rel}{\Rel} preserves $\Rel$-comprehension if it $\fn{F}$ preserves injections. 
Every co-monad over $\Set$ preserves injections so every co-monadic category over $\Set$ is an extensional comprehension completion. 
Indeed, by \cref{ex:hulls:span}, because $\Set$ is a boolean topos, it admits at least two $\Rel$-injective hulls: 
one is given by non-empty sets and the other one by powersets. 
\end{example}

\subsection{Comprehension-surjection factorization}
\label{sect:comp-factor} 

In \cite{DagninoP26jpaa} it is proved that, if $\RDoc$ is a relational doctrine with quotients, 
then $\RDoc$-quotient arrows are the left class of an orthogonal factorization system (on the base of $\RDoc$) whose right class are $\RDoc$-injections. 
In particular, the factorization of an arrow $\fun{f}{X}{Y}$ is obtained by quotienting $X$ by the kernel of $\ker_f=\gr{f}\rcomp\gr{f}\rconv$. 
By duality, we get the following result. 

\begin{proposition}\label{prop:fs-cq} 
Let $\fun{\RDoc}{\bop\CC}{\Pos}$ be a relational doctrine with comprehension, then $(\ct{S}_\RDoc, \ct{K}_\RDoc)$ is an orthogonal  factorisation system on \CC where $\ct{S}_\RDoc$ is the class of $\RDoc$-surjections and $\ct{K}_\RDoc$ is the class of $\RDoc$-comprehension arrows.
\end{proposition}

Here, the factorization of an arrow $\fun{f}{X}{Y}$ is obtained by taking the comprehension  of the image of $f$ over $Y$, \ie the relation $\img_f=\gr{f}\rconv\rcomp\gr{f}$.

%% file: qc-interaction.tex
\section{Interaction between quotients and comprehensions}
\label{sect:q-c-interaction}

Relational doctrines with both quotients an comprehensions  have many interesting properties. 
First of all, they are self-dual, in the sense that 
if $\RDoc$ has both quotients and comprehensions, then $\RDoc\op$ has both quotients and comprehensions. 
Furthermore, they enjoy a form of the first homomorphism theorem. 
Indeed, if $\RDoc$ has both quotients and comprehensions, by results in \cref{sect:comp-factor}, 
every arrow $\fun{f}{X}{Y}$ in the base of $\RDoc$ has two factorizations as in the following commutative diagram. 
\[\xymatrix{
& Y_{\img_f} \ar[rd]^-{k_f} \\
  X \ar[rr]^-{f} \ar[ur]^-{s_f} \ar[rd]_-{q_f} 
&& Y \\
& X/\ker_f \ar[ur]_-{i_f} 
}\] 
 where $\fun{q_f}{X}{X/\ker_f}$ is an $\RDoc$-quotient arrow for $\ker_f $, 
$\fun{k_f}{Y_{\img_f}}{Y}$ is an $\RDoc$-comprehension arrow for $\img_f$, and 
$s_f$ and $i_f$ are an $\RDoc$-surjection and an $\RDoc$-injection respectively.

\begin{proposition}
Let $\fun{\RDoc}{\bop\CC}{\Pos}$ be a relational doctrine with quotients  and comprehensions. 
Then, for every arrow $\fun{f}{X}{Y}$ in $\CC$, there exists a unique arrow $\fun{h_f}{X/\ker_f}{Y_{\img_f}}$ making the following diagram commute
\[\xymatrix{
& Y_{\img_f} \ar[rd]^-{k_f} \\
  X \ar[ur]^-{s_f} \ar[rd]_-{q_f} 
&& Y \\
& X/\ker_f \ar[uu]^-{h_f} \ar[ur]_-{i_f} 
}\] 
Moreover, $h_f$ is $\RDoc$-bijective. 
\end{proposition}
\begin{proof} 
Observe that the image of $i_f$ coincides with the image of $f$: using the fact that $q_f$ is surjective , as it is a quotient arrow, we have 
$ \gr{i_f}\rconv\rcomp\gr{i_f} 
  = \gr{i_f}\rconv\rcomp\gr{q_f}\rconv\rcomp\gr{q_f}\rcomp\gr{i_f} 
  = \gr{f}\rconv\rcomp\gr{f} 
$. 
Since $k_f$ is a comprehension arrow, we have a unique arrow $\fun{h_f}{X/\ker_f}{Y_{\img_f}}$  such that 
$k_f\circ h_f = i_f$. 
Precomposing with $q_f$ we get 
$k_f \circ h_f \circ q_f = i_f \circ q_f = f = k_f \circ s_f$, thus, because $k_f$ is monic, as it is a comprehension arrow, we  conclude that 
$h_f\circ q_f = s_f$, proving that the diagram commutes. 

Finally the equalities
\begin{align*}
\gr{h_f}\rcomp\gr{h_f}\rconv 
  &= \gr{h_f}\rcomp\gr{k_f}\rcomp\gr{k_f}\rconv\rcomp\gr{h_f}\rconv 
   = \gr{i_f}\rcomp\gr{i_f}\rconv 
   = \rid_Y \\
\gr{h_f}\rconv\rcomp\gr{h_f} 
  &= \gr{h_f}\rconv\rcomp\gr{q_f}\rconv\rcomp\gr{q_f}\rcomp\gr{h_f} 
   = \gr{s_f}\rconv\rcomp\gr{s_f}
   = \rid_X 
\end{align*}
shows that $h_f$ is $\RDoc$-bijective.
\end{proof} 

We regard this result as a form of the first homomorphism theorem as it states that 
the (comprehension of the) image of $f$ is  in bijection with the quotient of the domain by the kernel of $f$. 
This ensures that, from the point of view of relations, these two objects are indistinguishable. 
To turn the bijection into an actual isomorphism, we need an additional property of $\RDoc$: it has to be \emph{balanced} that means exactly that every $\RDoc$-bijective arrow is an isomorphism (see \cite{DagninoP26jpaa}). 

A natural question at this point is whether we can compose the two completions we have for quotients and comprehensions to add both of them to a relational doctrine.
From a 2-categorical point of view, this amount to showing that the two 2-monads $\QMnd$ and $\CMnd$ compose. 
In turn, this is equivalent to  prove either that $\QMnd$ restricts to $\CRDtn$ or that $\CMnd$ restricts to $\QRDtn$. 
Unfortunately, none of these to facts hold, as witnessed by the following examples. 

\begin{example}\label{ex:controcontro}
Consider the relational doctrine $\VRel\RPos$ of metric relations introduced in \cref{ex:rel-doc:vrel}. 
As noticed in \cref{ex:comp:vrel}, it has comprehensions because coequivalence relations on  a set $X$ correspond to its subsets. 
Now consider its intensional quotient completion $\RDoc = \QR{\VRel\RPos}$ that, as observed in \cite{DagninoP25apal}, it is the relational doctrine of bimodules over(pseudo)metric spaces. 
Consider an $\RDoc$-predicate $\preda$ over $\ple{X,\eqrelr}$ and let $\hat\preda \in \VRel\RPos(X,X)$ be 
\[\hat\preda(x,x') = \begin{cases}
0 &  x = x' \text{ and } \preda(x,x') = 0 \\
\infty & \text{otherwise} 
\end{cases}\]
Let $\fun{k_{\hat\preda}}{X_{\hat\preda}}{X}$ be a $\VRel\RPos$-comprehension arrow of $\hat\preda$ and consider  the $\VRel\RPos$-equivalence relation on $X_{\hat\preda}$ given by 
$\eqrelr_{\hat\preda} = \gr{k_{\hat\preda}}\rcomp\eqrelr\rcomp\gr{k_{\hat\preda}}\rconv$. 
Then, the arrow 
$\fun{k_{\hat\preda}}{\ple{X_{\hat\preda},\eqrelr_{\hat\preda}}}{\ple{X,\eqrelr}}$ is a well-defined arrow in the base of $\RDoc$. 
We claim that this is an $\RDoc$-comprehension arrow for $\preda$. 
Indeed, if $\fun{f}{\ple{Y,\eqrels}}{\ple{X,\eqrelr}}$ is an arrow such that $\eqrels \order \gr{f}\rcomp\preda\rcomp\gr{f}\rconv$, 
then we have that $\preda(f(y),f(y)) = 0$ for all $y \in Y$, as $0 = \eqrels(y,y) \geq \preda(f(y),f(y))$. 
Hence, this ensures that $\rid_Y\order \gr{f}\rcomp\hat\preda\rcomp\gr{f}\rconv$ and so, 
 because $k_{\hat\preda}$ is a $\VRel\RPos$-comprehension arrow for $\hat\preda$, we get a unique arrow $\fun{g}{Y}{X_{\hat\preda}}$ such that 
$k_{\hat\preda}\circ g = f$. 
To conclude, it remains to show that 
$\eqrels \order \gr{g}\rcomp\eqrelr_{\hat\preda}\rcomp\gr{g}\rconv = \gr{g}\rcomp\gr{k_{\hat\preda}}\rcomp\eqrelr\rcomp\gr{k_{\hat\preda}}\rconv\rcomp\gr{g}\rconv = \gr{f}\rcomp\eqrelr\rcomp\gr{f}\rconv$, which is immediate by definition of $f$. 

This shows that $\RDoc$ has strong comprehensions. However, these are not full in general. 
Consider a set $X = \{a,b,c\}$ and two functions $\fun{\eqrelr,\preda}{X\times X}{[0,\infty]}$ depicted below:
\[\begin{tikzpicture}[
    edge/.style={line width=0.8pt},
    vtx/.style={inner sep=1pt, outer sep=3pt},
    loop_a/.style={edge, looseness=8, distance=1.2cm, out=210, in=150},
    loop_b/.style={edge, looseness=8, distance=1.2cm, out=330, in=30},
    loop_c/.style={edge, looseness=8, distance=1.2cm, out=60,  in=120}
]
\node at (0, 2) {$\eqrelr$};

\node[vtx] (a1) at (0,0) {a};
\node[vtx] (b1) at (2,0) {b};
\node[vtx] (c1) at (1,1.73) {c};

\path (a1) edge[loop_a, "0" {left}] (a1);
\path (b1) edge[loop_b, "0" {right}] (b1);
\path (c1) edge[loop_c, "0" {above}] (c1);

\draw[edge] (a1) -- node[below] {1} (b1);
\draw[edge] (b1) -- node[above right] {2} (c1);
\draw[edge] (c1) -- node[above left] {3} (a1);

\begin{scope}[xshift=5.5cm]
\node at (0, 2) {$\preda$};

\node[vtx] (a2) at (0,0) {a};
\node[vtx] (b2) at (2,0) {b};
\node[vtx] (c2) at (1,1.73) {c};

\path (a2) edge[loop_a, "0" {left}] (a2);
\path (b2) edge[loop_b, "2" {right}] (b2);
\path (c2) edge[loop_c, "5" {above}] (c2);

\draw[edge] (a2) -- node[below] {1} (b2);
\draw[edge] (b2) -- node[above right] {4} (c2);
\draw[edge] (c2) -- node[above left] {3} (a2);
\end{scope}

\end{tikzpicture}\]
It is easy to see that $\eqrelr$ is a $\VRel\RPos$-equivalence relation on $X$, thus $\ple{X,\eqrelr}$ is an object in the base of $\RDoc$, 
and $\preda$ is an $\RDoc$-predicate on it. 
The $\RDoc$-comprehension arrow is 
$\fun{k_{\hat\preda}}{\ple{\{a\},\rid_{\{a\}}}}{\ple{X,\eqrelr}}$ where $k_{\hat\preda}$ is the inclusion of $\{a\}$ into $X$. 
Its image is the relation $\relr = \eqrelr \rcomp \hat\preda \rcomp \eqrelr$ on $X$ depicted below, which is clearly different from $\preda$. 
\[\begin{tikzpicture}[
    edge/.style={line width=0.8pt},
    vtx/.style={inner sep=1pt, outer sep=3pt},
    loop_a/.style={edge, looseness=8, distance=1.2cm, out=210, in=150},
    loop_b/.style={edge, looseness=8, distance=1.2cm, out=330, in=30},
    loop_c/.style={edge, looseness=8, distance=1.2cm, out=60,  in=120}
]

\node[vtx] (a1) at (0,0) {a};
\node[vtx] (b1) at (2,0) {b};
\node[vtx] (c1) at (1,1.73) {c};

\path (a1) edge[loop_a, "0" {left}] (a1);
\path (b1) edge[loop_b, "2" {right}] (b1);
\path (c1) edge[loop_c, "6" {above}] (c1);

\draw[edge] (a1) -- node[below] {1} (b1);
\draw[edge] (b1) -- node[above right] {4} (c1);
\draw[edge] (c1) -- node[above left] {3} (a1);

\end{tikzpicture}\] 
\end{example}

\cref{ex:controcontro} shows that  $\QMnd$ does not restrict to $\CRDtn$. Thanks to \cref{eq:contresempio} and \cref{prop:q-c-duality}, the same example shows also that also $\CMnd$ does not restrict to $\QRDtn$.

Looking at the above examples, it is pretty clear that the main issue is that the completions  modify identity relations  and, as a consequence, equivalences and predicates. 
Then, it is difficult, for instance, to turn an $\QR\RDoc$-predicate on $\ple{X,\eqrelr}$ into an $\RDoc$-predicate on $X$ so that one can use comprehensions in $\RDoc$  to  obtain those in $\QR\RDoc$. 
The rest of this section is devoted to the study of sufficient conditions ensuring that this happens. 

The first condition we consider is modularity. Recall from \cref{prop:pred-modular} that in a modular relational doctrine $\RDoc$, predicates coincide with coreflexive relations. Then, if we have an endorelation $\relr$ in $\RDoc(X,X)$, we can turn  it into a predicate by taking $\relr\land\rid_X$. 
The modularity condition gives to comprehensions many useful properties as described by the following proposition. 

\begin{proposition} \label[prop]{prop:c-mod-pullback}
Let $\fun{\RDoc}{\bop\CC}{\Pos}$ be a modular relational doctrine with comprehensions. 
For every arrow $\fun{f}{X}{Y}$ and $\RDoc$-comprehension arrow $\fun{k}{A}{Y}$, 
the pullback of $k$ along $f$ exists 
\[\xymatrix{
  P \ar[r]^-{f'} \ar[d]^-{k'}
& A \ar[d]^-{k} \\
  X \ar[r]^-{f} 
& Y 
}\]
where $\fun{k'}{P}{X}$ is an $\RDoc$-comprehension arrow 
and the following hold: 
\begin{enumerate} 
\item if $f$ is $\RDoc$-injective, $f'$ is $\RDoc$-injective, 
\item if $f$ is $\RDoc$-surjective, $f'$ is $\RDoc$-surjective, 
\item if $f$ is an $\RDoc$-comprehension arrow, $f'$ is an $\RDoc$-comprehension arrow. 
\end{enumerate}
\end{proposition}
\begin{proof}
Let $\preda$ be the predicate on $X$ defined by $\preda = (\gr{f}\rcomp\gr{k}\rconv\rcomp\gr{k}\rcomp\gr{f}\rconv)\land \rid_X$ and let $\fun{k'}{P}{X}$ be the comprehension of $\preda$. 
Observe that we  have 
$ \gr{f}\rconv\rcomp\gr{k'}\rconv\rcomp\gr{k'}\rcomp\gr{f} 
  \order \gr{f}\rconv\rcomp\gr{f}\rcomp\gr{k}\rconv\rcomp\gr{k}\rcomp\gr{f}\rconv\rcomp\gr{f} 
  \order \gr{k}\rconv\rcomp\gr{k}$, 
hence, because $k$ is a comprehension arrow, we get a unique arrow $\fun{f'}{P}{A}$ such that $k\circ f'  = f\circ k'$. 
To prove this is a pullback, consider two arrows $g$ and $h$ such that $k\circ h = f\circ g$. 
Since $\gr{g}\rconv\rcomp\gr{g} \order \rid_X$ and 
$ \gr{g}\rconv\rcomp\gr{g}
  \order \gr{f}\rcomp\gr{k}\rconv\rcomp\gr{h}\rconv\rcomp\gr{h}\rcomp\gr{k}\rcomp\gr{f}\rconv 
  \order \gr{f}\rcomp\gr{k}\rconv\rcomp\gr{k}\rcomp\gr{f}\rconv$, 
we derive $\gr{g}\rconv\rcomp\gr{g} \order \preda$ and so, because $k'$ is a comprehension arrow, we get a unique arrow $p$ such that $k'\circ p = g$. 
Then, it is easy to see that $k\circ f'\circ p = k\circ h$, hence, because $k$ is monic, we conclude also that $h = f'\circ p$, as needed. 

\begin{enumerate}
\item If $f$ is injective, we know that $\gr{f}\rcomp\gr{f}\rconv = \rid_X$. 
Since both $k$ and $k'$ are injective, as they are comprehension arrows,  from the equality $f\circ k' = k \circ f'$  we deduce 
\begin{align*}
\gr{f'}\rcomp\gr{f'}\rconv 
  &= \gr{f'}\rcomp\gr{k}\rcomp\gr{k}\rconv \rcomp\gr{f'}\rconv 
   = \gr{k'}\rcomp\gr{f} \rcomp \gr{f}\rconv\rcomp\gr{k}\rconv 
   = \rid_P 
\end{align*}
\item If $f$ is surjective, we know that $\gr{f}\rconv\rcomp\gr{f} = \rid_Y$. 
Again, since both $k$ and $k'$ are injective, we know that 
$\gr{f'} = \gr{k'}\rcomp\gr{f}\rcomp\gr{k}\rconv$. Then,  we have 
\begin{align*}
\gr{f'}\rconv\rcomp\gr{f'} 
  &= \gr{k}\rcomp\gr{f}\rconv\rcomp\gr{k'}\rconv\rcomp\gr{k'}\rcomp\gr{f}\rcomp\gr{k}\rconv  \\
  &= \gr{k}\rcomp\gr{f}\rconv\rcomp((\gr{f}\rcomp\gr{k}\rconv\rcomp\gr{k}\rcomp\gr{f}\rconv)\land\rid_X)\rcomp\gr{f}\rcomp\gr{k}\rconv \\
  &= \gr{k}\rcomp((\gr{k}\rconv\rcomp\gr{k})\land(\gr{f}\rconv\rcomp\gr{f}))\rcomp\gr{k}\rconv \\
  &= \gr{k}\rcomp((\gr{k}\rconv\rcomp\gr{k})\land\rid_Y)\rcomp\gr{k}\rconv \\
  &= \gr{k}\rcomp\gr{k}\rconv\rcomp\gr{k}\rcomp\gr{k}\rconv 
   = \rid_A
\end{align*}
where we have used the equality $\gr{f}\rconv\rcomp((\gr{f}\rcomp\relr\rcomp\gr{f}\rconv)\land\rels)\rcomp\gr{f} = \relr\land(\gr{f}\rconv\rcomp\rels\rcomp\gr{f})$, known as Frobenius reciprocity, which is a consequence of the modularity law (see \cite{DagninoP25apal}). 
\item If $f$ is a comprehension arrow, also the composition $f\circ k'$ is a comprehension arrow. 
Consider an arrow $g$ such that $\gr{g}\rconv\rcomp\gr{g}\order \gr{f'}\rconv\rcomp\gr{f'}$. 
Then, we have 
\begin{align*}
\gr{k}\rconv\rcomp\gr{g}\rconv\rcomp\gr{g}\rcomp\gr{k}
  &\order \gr{k}\rconv\rcomp\gr{f'}\rconv\rcomp\gr{f'}\rcomp\gr{k} 
   = \gr{f}\rconv\rcomp\gr{k'}\rconv\rcomp\gr{k'}\rcomp\gr{f} 
\end{align*}
Hence, we get a unique arrow $h$ such that $k\circ g = f\circ k'\circ h$. 
This implies that $k\circ g = k\circ f'\circ h$ and so, because $k$ is monic, we conclude 
$g = f'\circ h$, as needed. 
\end{enumerate}
\end{proof}

The next proposition shows that modularity is preserve by both the (intensional) quotient and the comprehension completions. 

\begin{proposition}
If $\RDoc$ is a modular relational doctrine, then both $\QR\RDoc$ and $\CR\RDoc$ are modular. 
\end{proposition}
\begin{proof}
It is enough to check that $\QR\RDoc(\ple{X,\eqrelr},\ple{Y,\eqrels})$ and $\CR\RDoc(\ple{X,\preda},\ple{Y,\predb})$ are closed under binary meets and have a top element preserved by reindexing.  
For $\QR\RDoc$ we have 
$\eqrelr\rcomp\top\rcomp\eqrels \order\top$ and 
$ \eqrelr \rcomp (\relr\land\rels) \rcomp \eqrels 
  \order (\eqrelr\rcomp\relr\rcomp\eqrels) \land (\eqrelr\rcomp\rels\rcomp\eqrels) 
  \order \relr\land\rels$, 
where $\relr,\rels \in \QR\RDoc(\ple{X,\eqrelr},\ple{Y,\eqrels})$. 

For $\CR\RDoc$, given $\relr,\rels\in\CR\RDoc(\ple{X,\preda},\ple{Y,\predb})$, we have 
$\relr \order \preda\rcomp\relr\rcomp\predb \order \preda\rcomp\top\rcomp\predb$, which thus is the top element, and 
$ \relr\land\rels 
  \order (\preda\rcomp\relr)\land (\rels\rcomp\predb) 
  \order \preda \rcomp ((\relr\rcomp\predb\rconv)\land(\preda\rconv\rcomp\rels))\rcomp\predb 
  \order \preda\rcomp(\relr\land\rels)\rcomp\predb$, 
as needed. 
\end{proof}

Finally, we show that on modular relational doctrines the (intensional) quotient completion preserves comprehensions. 

\begin{theorem} \label[thm]{thm:qc-modular} 
If $\RDoc$ is a modular relational doctrine with comprehensions, then 
$\QR\RDoc$ has comprehensions. 
\end{theorem}
\begin{proof}
Let $\preda$ be a $\QR\RDoc$-predicate on $\ple{X,\eqrelr}$. 
Then, $\hat\preda = \preda\land\rid_X$, by \cref{prop:pred-modular}, is an $\RDoc$-predicate on $X$. 
Let $\fun{k}{A}{X}$ be an $\RDoc$-comprehension arrow for $\hat\preda$ and consider the relation $\eqrelr_\preda = \gr{k}\rcomp\eqrelr\rcomp\gr{k}\rconv$, which is an $\RDoc$-equivalence relation on $X$. 
Then, the arrow $\fun{k}{\ple{A,\eqrelr_\preda}}{\ple{X,\eqrelr}}$ is a well-defined arrow in $\QC\RDoc$. 
We claim it is an $\QR\RDoc$-comprehension arrow for $\preda$. 
Consider an arrow $\fun{f}{\ple{Y,\eqrels}}{\ple{X,\eqrelr}}$ such that $\eqrels \order \gr{f}\rcomp\preda\rcomp\gr{f}\rconv$,  
then, since $\rid_Y\order\eqrels$, we derive 
$\gr{f}\rconv\rcomp\gr{f} \order \preda$ and, since $\gr{f}\rconv\rcomp\gr{f}\order \rid_X$, we conclude 
$\gr{f}\rconv\rcomp\gr{f}\order \preda\land\rid_X = \hat\preda$. 
Because $\fun{k}{A}{X}$ is a comprehension arrow, we get a unique $\fun{g}{Y}{A}$ such that $f = k\circ g$. 
Since we have $\eqrels \order \gr{f}\rcomp\eqrelr\rcomp\gr{f}\rconv = \gr{g}\rcomp \eqrelr_\preda \rcomp \gr{g}\rconv$, we conclude that 
$\fun{g}{\ple{Y,\eqrels}}{\ple{A,\eqrelr_\preda}}$ is a well-defined arrow in $\QC\RDoc$, as needed. 

The arrow $k$ is obviously $\QR\RDoc$-injective by construction. 
To prove it is a full comprehension arrow, we have to  show that $\preda$ is below the image of $k$ that is the relation 
$\eqrelr\rcomp\hat\preda\rcomp\eqrelr\rcomp\hat\preda\rcomp\eqrelr$. 
To this end, it suffices to check that 
$\preda\order \eqrelr\rcomp\hat\preda\rcomp\eqrelr$. 
Using the fact that $\preda$ is a descent datum ($\eqrelr\rcomp\preda\rcomp\eqrelr\order\preda$) and is coreflexive ($\preda\order\eqrelr$), we have the following 
\begin{align*}
\preda &\order \eqrelr\rcomp\preda 
       &\order \eqrelr\rcomp(\preda\land\eqrelr) 
       &\order \eqrelr\rcomp((\preda\rcomp\eqrelr\rconv)\land\rid_X)\rcomp\eqrelr \\
       &\order \eqrelr\rcomp(\preda\land\rid_X)\rcomp\eqrelr 
       &= \eqrelr\rcomp\hat\preda\rcomp\eqrelr 
\end{align*}
\end{proof}

These results ensure that the 2-monads $\QMnd$ and $\CMnd$ restrict to the full 2-subcaegory of $\RDtn$ on modular relational doctrines and on it the composition $\QMnd\circ\CMnd$ is again a 2-monad whose pseudoalgebras are modular relational doctrines with both quotients and comprehensions. 
This result applies to a wide range of examples including relational doctrines of spans over categories with weak finite limits and  
those coming from existential elementary doctrines. 

We now give a a more explicit description of the relational doctrine $\QR{\CR\RDoc}$ for a modular relational doctrine $\RDoc$. 
More precisely, we will show that the base category $\QC{\CR\RDoc}$ is isomorphic to a category of partial equivalence relations, thus generalizing known results obtained for existential elementary doctrines \cite{MaiettiME:eleqc} (which in turn generalises the construction of q-toposes from triposes given in \cite{Frey2023Categories}). 

Let $\fun\RDoc{\bop\CC}{\Pos}$ be a relational doctrine. 
The category $\PERC\RDoc$ of partial equivalence relations in $\RDoc$ is defined as follows: 
\begin{itemize}
\item an object in $\PERC\RDoc$ is a pair $\ple{X,\eqrelr}$ where $X$ is an object of $\CC$ and $\eqrelr$ is a partial equivalence relation on $X$, \ie it satisfies 
$\eqrelr\order\eqrelr\rconv$ and $\eqrelr\rcomp\eqrelr\order\eqrelr$, 
\item an arrow $\fun{f}{\ple{X\eqrelr}}{\ple{Y,\eqrels}}$ in $\PERC\RDoc$ is an arrow $\fun{f}{X}{Y}$ in $\CC$ such that 
$\eqrelr\order\gr{f}\rcomp\eqrels\rcomp\gr{f}\rconv$. 
\end{itemize}

\begin{proposition} \label[prop]{prop:per} 
Let $\fun\RDoc{\bop\CC}{\Pos}$ be a modular relational doctrine. 
Then, the categories $\PERC\RDoc$ and $\QC{\CR\RDoc}$ are isomorphic via a functor
$\fun{F}{\PERC\RDoc}{\QR{\CR\RDoc}}$ such that 
$F\ple{X,\eqrelr} = \ple{\ple{X,\eqrelr\land\rid_X},\eqrelr}$. 
\end{proposition}
\begin{proof}
First observe that, since $\RDoc$ is modular, $\eqrelr\land\rid_X$ is an $\RDoc$-predicate on $X$ and so $\ple{X,\eqrelr\land\rid_X}$ is an object in  the base of $\CR\RDoc$. 
Moreover, we have 
$ \eqrelr 
  = \eqrelr\land\eqrelr 
  \order \eqrelr\rcomp((\eqrelr\rconv\rcomp\eqrelr)\land\rid_X) 
  \order \eqrelr\rcomp(\eqrelr\land\rid_X)$. 
Since both $\eqrelr$ and $\eqrelr\land\rid_X$ are symmetric, this also implies that 
$\eqrelr\order (\eqrelr\land\rid_X)\rcomp\eqrelr$, 
thus proving that $\eqrelr$ is a codescent datum for $\eqrelr\land\rid_X$. 
Finally, since $\eqrelr\land\rid_X\order\eqrelr$, we hav ethat 
$\eqrelr$ is a $\CR\RDoc$-equivalence relation on $\ple{X,\eqrelr\land\rid_X}$ and so 
$\ple{\ple{X,\eqrelr\land\rid_X},\eqrelr}$ is a well-defined object of $\QC{\CR\RDoc}$. 

It is easy to verify that the assignment $F\ple{X\eqrelr} = \ple{\ple{X,\eqrelr\land\rid_X},\eqrelr}$ extends to a fully faithful functor 
$\fun{F}{\PERC\RDoc}{\QC{\CR\RDoc}}$. 
Therefore, to conclude it suffices to check that 
for every object $\ple{\ple{X,\preda},\eqrelr}$ in $\QC{\CR\RDoc}$, there is a partial $\RDoc$-equivalence relation $\eqrelr'$ on $X$ such that 
$F\ple{X,\eqrelr'} = \ple{\ple{X,\preda},\eqrelr}$, that is, $\eqrelr' = \eqrelr$ and $\preda = \eqrelr\land\rid_X$. 
First, observe that, since $\eqrelr$ is a $\CR\RDoc$-equivalence relation, it is symmetric and transitive and so 
it is also a partial $\RDoc$-equivalence relation. 
Hence, it remains to show that $\preda = \eqrelr\land\rid_X$. 
To this end, note that, 
because $\preda$ is coreflexive in $\RDoc$ and $\eqrelr$ is reflexive in $\CR\RDoc$, we have 
$\preda\order\rid_X$ and $\preda\order\eqrelr$, thus $\preda\order\eqrelr\land\rid_X$. 
Conversely, since $\eqrelr$ is a codescent datum, we have 
$ \rid_X\land\eqrelr
  \order \rid_X \land (\eqrelr\rcomp\preda) 
  \order (\preda\rconv\land\eqrelr)\rcomp \preda 
  = \preda\rcomp\preda 
  =\preda$, 
because $\preda\order\eqrelr$ and $\preda = \preda\rconv$ imply $\preda\rconv\land\eqrelr = \preda$. 
\end{proof}

This allows us to define by change-of-base a relational doctrine $\PERR\RDoc$ on $\PERC\RDoc$ that will be isomorphic to $\QR{\CR\RDoc}$. 
Explicitly, a relation in $\PERR\RDoc(\ple{X,\eqrelr},\ple{Y,\eqrels})$ is a relation $\relr$ in $\RDoc(X,Y)$ such that 
$\eqrelr\rcomp\relr\rcomp\eqrels \order\relr$ and $\relr \order (\eqrelr\land\rid_X)\rcomp\relr\rcomp(\eqrels\land\rid_Y)$, 
relational composition and converse are those of $\RDoc$ and the identity relation on $\ple{X,\eqrelr}$ is $\rid_{\ple{X,\eqrelr}} = \eqrelr$.

We proceed by studying 
another sufficient condition that will ensure that we can compose the completions the other way round. 
We say that a relational doctrine $\RDoc$ is \emph{(join) distributive} if 
every fiber has finite joins and relational composition distributes over them, that is, 
for all relations $\relr,\rels,\relt_1,\relt_2$ 
\[
\relr\rcomp\bot\rcomp\rels = \bot 
\qquad 
\relr\rcomp(\relt_1\lor\relt_2)\rcomp\rels = (\relr\rcomp\relt_1\rcomp\rels)\lor(\relr\rcomp\relt_2\rcomp\rels)
\]
Note that, because both reindexing and its left adjoint can be expressed  using relational composition, 
this condition ensures that both of them preserve finite joins. 
It is easy to see that in a distributive relational doctrine, if $\relr$ is a symmetric and transitive relation on $X$, 
then $\relr\lor\rid_X$ is an equivalence relation on $X$. 

In a distributive relational doctrine, we can prove  a result which is close to a dual of \cref{prop:c-mod-pullback}, even if it is actually weaker. 

\begin{proposition}
Let $\fun{\RDoc}{\bop\CC}{\Pos}$ be a distributive relational doctrine with quotients. 
For every $\RDoc$-injective arrow $\fun{f}{X}{Y}$ and every $\RDoc$-quotient arrow $\fun{q}{X}{W}$, 
the pushout of $q$ along $f$ exists 
\[\xymatrix{
  X \ar[r]^-{f} \ar[d]^-{q} 
& Y \ar[d]^-{q'} \\
  W \ar[r]^-{f'}
& P 
}\]
where $\fun{q'}{Y}{P}$ is an $\RDoc$-quotient arrow and $\fun{f'}{W}{P}$ is $\RDoc$-injective. 
\end{proposition}
\begin{proof}
Consider the relation $\relr \in \RDoc(Y,Y)$ given by $\gr{f}\rconv\rcomp\gr{q}\rcomp\gr{q}\rconv\rcomp\gr{f}$. 
It is clearly symmetric and, because $f$ is injective, it is also transitive, hence 
$\eqrelr = \relr \lor \rid_Y$ is an equivalence relation on $Y$. 
Let $\fun{q'}{Y}{P}$ be a quotient arrow of $\eqrelr$. 
Since we have 
\begin{align*}
\gr{q}\rcomp\gr{q}\rconv 
  &\order \gr{f}\rcomp\gr{f}\rconv\rcomp\gr{q}\rcomp\gr{q}\rconv\rcomp\gr{f}\rcomp\gr{f}\rconv 
   = \gr{f}\rcomp\relr\rcomp\gr{f}\rconv  \\  
  &\order \gr{f}\rcomp\eqrelr\rcomp\gr{f}\rconv 
   = \gr{f}\rcomp\gr{q'}\rcomp\gr{q'}\rconv\rcomp\gr{f}\rconv
\end{align*}
we get a unique arrow 
$\fun{f'}{W}{P}$ such that $f'\circ q = q'\circ f$, because $q$ is a quotient arrow. 

To prove this commutative square is a pushout, consider two arrow $g$ and $h$  such that 
$g\circ f = h\circ q$ and observe that 
$\relr \order \gr{f}\rcomp\gr{h}\rconv\rcomp\gr{h}\rcomp\gr{g}\rconv \order \gr{g}\rcomp\gr{g}\rconv$.
Then, since $\gr{g}\rcomp\gr{g}\rconv$ is reflexive, we derive 
$\eqrelr\order \gr{g}\rcomp\gr{g}\rconv$ and so, because $q'$ is a quotient arrow, we get a unique arrow  $p$ such that 
$g = p\circ q'$. 
Then, it is easy to see that 
$h\circ q = p \circ f' \circ q$ and so, because $q$ is epic, we conclude that $h = p\circ f'$ as needed. 

Finally, to verify that $f'$ is injective, observe that, because $q$ is surjective, we have 
$\gr{f'}= \gr{q}\rconv\rcomp\gr{f}\rcomp\gr{q'}$. 
Then, we deduce 
\begin{align*}
\gr{f'}\rcomp\gr{f'}\rconv 
  &= \gr{q}\rconv\rcomp\gr{f}\rcomp\gr{q'}\rcomp\gr{q'}\rconv\rcomp\gr{f}\rconv\rcomp\gr{q} 
   = \gr{q}\rconv\rcomp\gr{f}\rcomp(\relr\lor\rid_Y)\rcomp\gr{f}\rconv\rcomp\gr{q} \\
  &= (\gr{q}\rconv\rcomp\gr{f}\rcomp\relr\rcomp\gr{f}\rconv\rcomp\gr{q}) \lor 
     (\gr{q}\rconv\rcomp\gr{f}\rcomp\gr{f}\rconv\rcomp\gr{q}) \\
  &= (\gr{q}\rconv\rcomp\gr{f}\rcomp\gr{f}\rconv\rcomp\gr{q}\rcomp\gr{q}\rconv\rcomp\gr{f}\rcomp\gr{f}\rconv\rcomp\gr{q}) \lor 
     (\gr{q}\rconv\rcomp\gr{f}\rcomp\gr{f}\rconv\rcomp\gr{q}) \\
  &= (\gr{q}\rconv\rcomp\gr{q}\rcomp\gr{q}\rconv\rcomp\gr{q}) \lor (\gr{q}\rconv\rcomp\gr{q}) 
   = \rid_W
\end{align*}
\end{proof}

The next proposition shows that distributivity is preserve by both the (intensional) quotient and comprehension completions. 

\begin{proposition}
If $\RDoc$ is a distributive relational doctrine, then both $\QR\RDoc$ and $\CR\RDoc$ are distributive. 
\end{proposition}
\begin{proof}
It is enough to check that $\QR\RDoc(\ple{X,\eqrelr},\ple{Y,\eqrels})$ and $\CR\RDoc(\ple{X,\preda},\ple{Y,\predb})$ are closed under binary joins and have a bottom element preserved by the relational composition. 
For $\CR\RDoc$ we have 
$\bot\order\preda\rcomp\bot\rcomp\predb$ and 
$ \relr\lor\rels 
  \order (\preda\rcomp\relr\rcomp\predb) \lor (\preda\rcomp\rels\rcomp\predb) 
  = \preda\rcomp(\relr\lor\rels)\rcomp\predb$, 
where $\relr,\rels \in \CR\RDoc(\ple{X,\preda},\ple{Y,\predb})$. 

For $\QR\RDoc$, given $\relr,\rels\in\QR\RDoc(\ple{X,\eqrelr},\ple{Y,\eqrels})$, we have 
$\eqrelr\rcomp\bot\rcomp\eqrels \order \eqrelr\rcomp\relr\rcomp\eqrels \order \relr$, which thus is the bottom element, and 
$ \eqrelr\rcomp(\relr\lor\rels)\rcomp\eqrels 
  = (\eqrelr\rcomp\relr\rcomp\eqrels) \lor (\eqrelr\rcomp\rels\rcomp\eqrels) 
  \order \relr\lor\rels$, 
as needed. 
\end{proof}

Finally, we show that on distributive relational doctrines the (intensional) comprehension completion preserves quotients. 

\begin{theorem} \label[thm]{thm:qc-distributive} 
If $\RDoc$ is a distributive relational doctrine with quotients, then 
$\CR\RDoc$ has quotients. 
\end{theorem}
\begin{proof}
Consider a $\CR\RDoc$-equivalence relation $\eqrelr$ on $\ple{X,\preda}$. 
By definition, it is a symmetric and transitive relation on $X$, hence $\hat\eqrelr = \eqrelr\lor\rid_X$ is an $\RDoc$-equivalence relation on $X$. 
Let $\fun{q}{X}{W}$ be an $\RDoc$-quotient arrow for $\hat\eqrelr$. 
The relation $\preda_\eqrelr = \gr{q}\rconv\rcomp\preda\rcomp\gr{q}$ is an $\RDoc$-predicate on $W$ and so the arrow 
$\fun{q}{\ple{X,\preda}}{\ple{W,\preda_\eqrelr}}$ is a well-defined arrow in $\CCt\RDoc$. 
We claim it is an $\CR\RDoc$-quotient arrow for $\eqrelr$. 
Let $\fun{f}{\ple{X,\preda}}{\ple{Y,\predb}}$ be an arrow in $\CCt\RDoc$ such that $\eqrelr\order\gr{f}\rcomp\gr{f}\rconv$. 
Then, because $\gr{f}\rcomp\gr{f}\rconv$ is reflexive, we have that $\hat\eqrelr\order\gr{f}\rcomp\gr{f}\rconv$ and so, because $q$ is an $\RDoc$-quotient arrow, we get a unique arrow $\fun{h}{W}{Y}$ such that $f = h\circ q$. 
Notice that, since $\preda\order\gr{f}\rcomp\predb\rcomp\gr{f}\rconv$, we deduce that 
$ \preda_\eqrelr 
  = \gr{q}\rconv\rcomp\preda\rcomp\gr{q}
  \order \gr{h}\rcomp\predb\rcomp\gr{h}\rconv$, 
hence $\fun{h}{\ple{W,\preda_\eqrelr}}{\ple{Y,\predb}}$ is a well-defined arrow in $\CCt\RDoc$,  as needed. 

Clearly, $q$ is $\CR\RDoc$-surjective by definition of $\preda_\eqrelr$. 
To check that it is effective, observe that 
the kernel of $q$ is the relation $\preda\rcomp\hat\eqrelr\rcomp\preda\rcomp\hat\eqrelr\rcomp\preda$ and, 
recalling that $\eqrelr$ is reflexive  on $\ple{X,\preda}$ ($\preda\order\eqrelr$), transitive ($\eqrelr\rcomp\eqrelr\order\eqrelr$), and 
relational composition distributes over joins,  we get that  this kernel is below $\eqrelr$ as needed. 
\end{proof}

These results ensure that the 2-monads $\QMnd$ and $\CMnd$ restrict to the full 2-subcategory of $\RDtn$ on distributive relational doctrines and on it 
the composition $\CMnd\circ\QMnd$  is again a 2-monad whose pseudoalgebras are distributive relational doctrines with both quotients and comprehensions. 
Also this condition covers a wide range of examples. 
Notably, all doctrines $\VRel\Qtl$ of $\Qtl$-relations over $\Set$ are distributive and also existential elementary doctrines with disjunctions have the same property. 

Notice that all the results presented so far in this section can be dualized obtaining other sufficient conditions ensuring a good interaction between quotients and comprehensions.
Notably, we have that, 
if the relational doctrine has fibered joins satisfying a dual version of the modularity law, \ie $\relr\rcomp(\rels\lor(\relr\rconv\rcomp\relt))\order (\relr\rcomp\rels)\lor\relt$, then  the quotient completion preserves comprehensions, and, 
if the relational doctrine is meet distributive, \ie the relational composition distributes over finite meets, then 
the comprehension completion preserves quotients.

We conclude this section by comparing $\QR{\CR\RDoc}$ and $\CR{\QR\RDoc}$ when $\RDoc$ is both modular and distributive. 
Recall from \cref{prop:per}  that $\QR{\CR\RDoc}$ is isomorphic to the relational doctrine $\PERR\RDoc$ of partial equivalence relations, 
hence we can safely use this latter for our comparison. 

We start by defining a 1-arrow $\oneAr{F}{\PERR\RDoc}{\CR{\QR\RDoc}}$. 
The functor $\fun{\fn{F}}{\PERC\RDoc}{\CCt{\QR\RDoc}}$ on objects is given by 
$\fn{F}\ple{X,\eqrelr} = \ple{\ple{X,\eqrelr\lor\rid_X},\eqrelr}$ and it is the identity on arrows. 
This is trivially well defined as, being $\eqrelr$ a partial $\RDoc$-equivalence relation on $X$, 
$\eqrelr\lor\rid_X$ is an $\RDoc$-equivalence relation on $X$ (because $\RDoc$ is distributive) and 
$\eqrelr$ is a $\QR\RDoc$-predicate on $\ple{X,\eqrelr\lor\rid_X}$. 
The components of the natural transformation $\lift{F}$ are given by the inclusion: 
if $\relr\in\PERR\RDoc(\ple{X,\eqrelr},\ple{Y,\eqrels})$ we have that 
$\eqrelr\rcomp\relr\rcomp\eqrels \order\relr$ implies $(\eqrelr\lor\rid_X)\rcomp\relr\rcomp(\eqrels\lor\rid_Y)\order\relr$ and 
$\relr\order(\eqrelr\land\rid_X)\rcomp\relr\rcomp(\eqrels\land\rid_Y)$ implies $\relr\order\eqrelr\rcomp\relr\rcomp\eqrels$. 
Recall that a 1-arrow in $\RDtn$   is \emph{fully faithful} if the underlying functor is fully faithful and the natural transformation is an isomorphism. Then, we can prove the following. 

\begin{proposition}
Let $\RDoc$ be a modular and distributive relational doctrine. 
The 1-arrow $\oneAr{F}{\PERR\RDoc}{\CR{\QR\RDoc}}$ is fully faithful and preserves quotients and comprehensions. 
\end{proposition}
\begin{proof}
The only non trivial part is that $\lift{F}$ is an isomorphism.
To this end, it suffices to show that 
given $\relr\in\CR{\QR\RDoc}(\fn{F}\ple{X,\eqrelr},\fn{F}\ple{Y,\eqrels})$ we have 
$\relr\in\PERR\RDoc(\ple{X,\eqrelr},\ple{Y,\eqrels})$. 
The fact that $\eqrelr\rcomp\relr\rcomp\eqrels\order\relr$ follows  immediately from 
$(\eqrelr\lor\rid_X)\rcomp\relr\rcomp(\eqrels\lor\rid_Y)\order\relr$ because $\RDoc$ is distributive. 
The fact that $\relr\order(\eqrelr\land\rid_X)\rcomp\relr\rcomp(\eqrels\land\rid_Y)$  follows from $\relr\order\eqrelr\rcomp\relr\rcomp\eqrels$ 
by observing that, 
by the modularity law, we have $\eqrelr\order\eqrelr\land\eqrelr\order(\eqrelr\land\rid_X)\rcomp\eqrelr$ and similarly for $\eqrels$. 
\end{proof}

We can also define a 1-arrow $\oneAr{G}{\CR{\QR\RDoc}}{\PERR\RDoc}$ where 
$\fn{G}\ple{\ple{X,\eqrelr},\preda} = \ple{X,\preda}$ and $\fn{G}(f) = f$ and 
the components of $\lift{G}$ are inclusions. 
The functor $\fn{G}$ is well defined because $\preda$, being an $\QR\RDoc$-predicate, it is  a partial $\RDoc$-equivalence relation. 
Indeed, we have $\preda\rcomp\preda \order \eqrelr\rcomp\preda\order\preda$. 
Moreover, $\lift{G}$ is well defined because, 
given $\relr\in\CR{\QR\RDoc}(\ple{\ple{X,\eqrelr},\preda},\ple{\ple{Y,\eqrels},\predb})$, we have 
$\preda\rcomp\relr\rcomp\predb \order \eqrelr\rcomp\relr\rcomp\eqrels \order\relr$ and 
\begin{align*} 
\relr 
  & \order \preda\rcomp\relr\rcomp\predb 
    \order (\preda\land\rid_X)\rcomp\eqrelr\rcomp\relr\rcomp\eqrels\rcomp(\predb\land\rid_Y)  \\ 
  & \order (\preda\land\rid_X)\rcomp\relr\rcomp(\predb\land\rid_Y)
\end{align*} 
because $\preda \order \preda\land\eqrelr \order ((\preda\rcomp\eqrelr\rconv)\land\rid_X)\rcomp\eqrelr \order (\preda\land\rid_X)\rcomp\eqrelr$ by the modularity law.   
Note that, differently from $F$, in general  $G$ is not fully faithful: 
neither $\fn{G}$ is full nor $\lift{G}$ is an isomorphism. 

It is easy to see that the composition 
$G\circ F$ is the identity on $\PERR\RDoc$. 
On the other hand, we have a 2-arrow 
$\twoAr{\theta}{F\circ G}{\Id_{\CR{\QR\RDoc}}}$ where the component at $\ple{\ple{X,\eqrelr},\preda}$ is the arrow 
$\fun{\id_X}{\ple{\ple{X,\preda\lor\rid_X},\preda}}{\ple{\ple{X,\eqrelr},\preda}}$. 
Hence, the components of $\theta$ are $\CR{\QR\RDoc}$-bijective. 
Alltogether, these observations prove the following result.

\begin{theorem}
Let $\RDoc$ be a modular and distributive relational doctrine. 
Then, $\oneAr{F}{\PERR\RDoc}{\CR{\QR\RDoc}}$ is a left adjoint section of $\oneAr{G}{\CR{\QR\RDoc}}{\PERR\RDoc}$ in $\RDtn$ where the components of the counit are $\CR{\QR\RDoc}$-bijective. 
\end{theorem}

This shows that for modular and distributive relational doctrines the compositions $\CMnd\circ\QMnd$ and $\QMnd\circ\CMnd$ are very close to be equivalent. 
Indeed, if we further apply the completion presented in \cite{DagninoP25tac}, which forces bijections to be isomorphisms, 
the two compositions become equivalent.

%% file: conclu.tex
\section{Concluding remarks}
\label{sect:conclu}

In this paper we have shown how, by shifting the perspective from usual predicate logic to the calculus of relations, one can recover a full duality between quotients and comprehensions. 
We have developed our results in the framework of relational doctrines, describing a duality for them that directly exchanges quotients for equivalence relations with comprehensions for coequivalence relations, which we claim to be a good representation of predicates in the relational setting.
We demonstrated that this correspondence lifts to a 2-dual isomorphism between the respective 2-categories, enabling the direct transfer of structural results, such as completions, 2-monadicity, characterizations via injective and projective objects, and factorization systems, between quotients and comprehensions.
Finally, we investigated their interaction, proving a version of the first isomorphism theorem and analyzing the conditions under which quotient and comprehension completions compose.

An intriguing parallel to our work appears in the series of papers on duality in non-abelian algebra by Janelidze, Weighill, and Goswami~\cite{janelidze2014duality1,janelidze2016duality2,janelidze2017duality3,goswami2019duality4}. 
There, the authors develop a self-dual categorical framework for non-abelian algebraic contexts, such as groups and regular varieties. 
By abstracting properties of the bifibration of subgroups over groups, they introduce concepts like quotients of normal subobjects and embeddings of conormal subobjects, which turn out to be dual to each other. 
These clearly resemble our notions of quotients and comprehensions for (co)equivalence relations, with the major difference given by the fact that they take an algebraic point of view, focusing on constructions for subobjects, while we have a logical perspective, rooted in the calculus of relations.
Drawing a formal, precise connection between our relational duality and their algebraic duality is a promising avenue for ongoing and future research.